\documentclass[11pt]{amsart} \textwidth=14.5cm \oddsidemargin=1cm
\usepackage{amsmath,wasysym}
\usepackage{amsxtra}
\usepackage{amscd}
\usepackage{amsthm}
\usepackage{amsfonts}
\usepackage{amssymb}
\usepackage{eucal}
\usepackage{graphics, color}
\usepackage{hyperref,mathrsfs}
\usepackage{verbatim}  % 需使用  verbatim package

\input prepictex
\input pictex
\input postpictex
\newtheorem{thm}{Theorem}[section]

\newtheorem{lem}[thm]{Lemma}
\newtheorem{cor}[thm]{Corollary}

\newtheorem{conj}[thm]{Conjecture}

\theoremstyle{definition}

\newtheorem{example}[thm]{Example}

\theoremstyle{remark}
\newtheorem{rem}[thm]{Remark}

\numberwithin{equation}{section}

\begin{document}

%Referring commands:
\newcommand{\thmref}[1]{Theorem~\ref{#1}}
\newcommand{\secref}[1]{Section~\ref{#1}}
\newcommand{\lemref}[1]{Lemma~\ref{#1}}
\newcommand{\propref}[1]{Proposition~\ref{#1}}
\newcommand{\corref}[1]{Corollary~\ref{#1}}
\newcommand{\remref}[1]{Remark~\ref{#1}}
\newcommand{\eqnref}[1]{(\ref{#1})}
\newcommand{\exref}[1]{Example~\ref{#1}}

%Simplified symbols:
\newcommand{\nc}{\newcommand}
\nc{\Z}{{\mathbb Z}}
%\nc{\hZ}{{\hf+\mathbb Z}}
\nc{\hZ}{{\underline{\mathbb Z}}}
\nc{\C}{{\mathbb C}}
\nc{\N}{{\mathbb N}}
\nc{\F}{{\mf F}}
\nc{\Q}{\mathbb {Q}}
\nc{\la}{\lambda}
\nc{\ep}{\delta}
\nc{\h}{\mathfrak h}
\nc{\n}{\mf n}
\nc{\A}{{\mf a}}
\nc{\G}{{\mathfrak g}}
\nc{\SG}{\overline{\mathfrak g}}
\nc{\DG}{\widetilde{\mathfrak g}}
\nc{\D}{\mc D}
\nc{\Li}{{\mc L}}
\nc{\La}{\Lambda}
\nc{\is}{{\mathbf i}}
\nc{\V}{\mf V}
\nc{\bi}{\bibitem}
\nc{\NS}{\mf N}
\nc{\dt}{\mathord{\hbox{${\frac{d}{d t}}$}}}
\nc{\E}{\EE}
\nc{\ba}{\tilde{\pa}}
\nc{\half}{\frac{1}{2}}
\nc{\mc}{\mathcal}
\nc{\mf}{\mathfrak} \nc{\hf}{\frac{1}{2}}
\nc{\hgl}{\widehat{\mathfrak{gl}}}
\nc{\gl}{{\mathfrak{gl}}}
\nc{\hz}{\hf+\Z}
\nc{\dinfty}{{\infty\vert\infty}}
\nc{\SLa}{\overline{\Lambda}}
\nc{\SF}{\overline{\mathfrak F}}
\nc{\SP}{\overline{\mathcal P}}
\nc{\U}{\mathfrak u}
\nc{\SU}{\overline{\mathfrak u}}
\nc{\ov}{\overline}
\nc{\wt}{\widetilde}
\nc{\wh}{\widehat}
\nc{\sL}{\ov{\mf{l}}}
\nc{\sP}{\ov{\mf{p}}}
\nc{\osp}{\mf{osp}}
\nc{\spo}{\mf{spo}}
\nc{\hosp}{\widehat{\mf{osp}}}
\nc{\hspo}{\widehat{\mf{spo}}}
\nc{\hh}{\widehat{\mf{h}}}
\nc{\even}{{\bar 0}}
\nc{\odd}{{\bar 1}}
\nc{\mscr}{\mathscr}

\newcommand{\blue}[1]{{\color{blue}#1}}
\newcommand{\red}[1]{{\color{red}#1}}
\newcommand{\green}[1]{{\color{green}#1}}
\newcommand{\white}[1]{{\color{white}#1}}
%%

%% added by WW

\newcommand{\aaf}{\mathfrak a}
\newcommand{\bb}{\mathfrak b}
\newcommand{\cc}{\mathfrak c}
\newcommand{\qq}{\mathfrak q}
\newcommand{\qn}{\mathfrak q (n)}
\newcommand{\UU}{\bold U}
\newcommand{\VV}{\mathbb V}
\newcommand{\WW}{\mathbb W}
\nc{\TT}{\mathbb T}
\nc{\EE}{\mathbb E}
\nc{\FF}{\mathbb F}
\nc{\KK}{\mathbb K}
\nc{\rl}{\texttt{r,l}}

 \advance\headheight by 2pt

\title[Reduction method for representations of queer Lie superalgebras]
{Reduction method for representations of queer Lie superalgebras}

\author[Chen]{Chih-Whi Chen}
\address{Department of Mathematics, National Taiwan University, Taipei 10617, Taiwan}
\email{d00221002@ntu.edu.tw}

\begin{abstract}
We develop a reduction procedure which provides an equivalence from an arbitrary  block of the BGG category for the queer Lie superalgebra $\mf{q}(n)$ to a "$\mathbb{Z}\pm s$-weights" ($s\in \mathbb{C}$)  block of a BGG category for finite direct sum of queer Lie superalgebras. We give  descriptions of blocks. We also establish equivalences between certain maximal parabolic subcategories for $\mf{q}(n)$ and  blocks of atypicality-one of the category of finite-dimensional modules for $\mf{gl}(\ell|n-\ell)$.

\end{abstract}

%\subjclass[2010]{17B67}

\maketitle

\let\thefootnote\relax\footnotetext{{\em Key words and phrases:}  Queer Lie superalgebra, general linear Lie superalgebra, BGG category, block decomposition,  maximal parabolic subcategory, equivalence.}

\let\thefootnote\relax\footnotetext{{\em 2010 Mathematics Subject Classification.} Primary 17B10.}

\section{Introduction} \label{Introduction}

\subsection{}
The character problem for finite-dimensional irreducible modules over queer Lie superalgebras $\mf{q}(n)$ was first solved by Penkov and Serganova \cite{PS1, PS2}. They provided an algorithm for computing the coefficient $a_{\la\mu}$ of the character of the irreducible $\mf{q}(n)$-module $L(\mu)$ in the expansion of the character of the associated Euler characteristic $E(\la)$ for given dominant weights $\la, \mu$.

In \cite{Br2} Brundan developed a different approach to computing the coefficient $a_{\la\mu}$ for integer dominant weights $\la,\mu$. Let $\mathbb{F}^n$ be the nth exterior power of the natural representation of type B quantum group $\text{U}_{{q}}(\mf{b}_{\infty})$ with infinite rank (cf. \cite{JMO}). It was proved that the transition matrix $(a_{\la\mu})$ is given by the transpose of the transition matrix between canonical basis and the natural monomial basis of $\mathbb{F}^n$ at $q=1$. This gives all irreducible characters of finite-dimensional integer weight modules in terms of a combinatorial algorithm for computing canonical bases. A new interpretation of the irreducible characters of finite-dimensional half-integer weight modules was given in \cite{CK} and \cite{CKW} as well.

The celebrated Brundan's Kazhdan-Lusztig conjecture \cite{Br1} for the BGG category of integer weight $\mf{gl}(m|n)$-modules has been proved by Cheng, Lam and Wang in \cite{CLW} (also see \cite{BLW}).  Furthermore, in \cite{CMW}, by using twisting functors and parabolic induction functors Cheng, Mazorchuk and Wang  reduced the irreducible character problem of an arbitrary weight to the problem of integer weight.

In the present paper, we study the problem analogous to \cite{CMW} for the queer Lie superalgebra. One of the main goals is to study the (indecomposable) blocks of the BGG category for queer Lie superalgebra. In particular, we will prove  equivalence of categories between certain blocks for  $\mf{q}(n)$ and $\mf{gl}(\ell|n-\ell)$.

Throughout the paper we denote by $\mathfrak{g}$ the queer Lie superalgebra $\mathfrak{q}(n)$ with the standard Cartan subalgebra $\mathfrak{h}$  for a fixed integer $n\geq 1$.  Let $\mc{O}^{\mf g}$ denote the BGG category (see, e.g., \cite[Section 3]{Fr}) of finitely generated $\mf{g}$-modules which are locally finite $\mf{b}$-modules and semisimple $\mf{h}_{\bar{0}}$-modules. Note that morphisms in $\mc{O}^{\mf{g}}$ are even. For a finite direct sum of queer Lie superalgebras and reductive Lie algebras, we have analogous notation of its BGG category. Let $m\in \mathbb{Z}_{+}$, $0\leq \ell \leq m$ and $s\in \mathbb{C}$. If $m\geq 1$, let $\Lambda_{s^{\ell}}(m)\subset \mathbb{C}^{m}$:
\begin{align} \label{DefOfLambda} \Lambda_{s^{\ell}}(m):=  \left\{ \lambda = (\lambda_1,\ldots,\lambda_m)  \mid
 \begin{array}{ll} (1) \ \ \lambda_i  \equiv s  \text{ mod } \mathbb{Z} \text{ for } 1\leq i\leq \ell,\\
(2)  \ \ \lambda_i  \equiv -s  \text{ mod } \mathbb{Z} \text{ for } \ell +1\leq i\leq n.\\
\end{array} \right\}. \end{align}
We define $\mf{q}(0)$ and $\Lambda_{s^{\ell}}(0)$ to be $0$ and the empty set, respectively. For each $\la \in \mf{h}_{\bar{0}}^*$, we shall assign a specific irreducible module $L(\la)$ of highest weight $\la$ and then define the corresponding block $\mc{O}_{\la}^{\mf{g}}$, see the definitions in Section \ref{CatOfModule}.  The following theorem is the first main result of this paper.

\begin{thm} \label{FirstMainThm}
 Let $\lambda \in \mathfrak{h}_{\bar{0}}^{*}$. Then $\mathcal{O}^{\mathfrak{g}}_{\lambda}$ is equivalent to a block $\mathcal{O}^{\mathfrak{l}}_{\mu}$ of a Levi subalgebra $\mathfrak{l}= \mathfrak{q}(n_1)\times \mathfrak{q}(n_2)\times \cdots \times \mathfrak{q}(n_{k}) \subseteq \mathfrak{g}$ with $\sum_{i=1}^k n_i = n$ and the weight $\mu$ of the form \begin{align}
 \mu\in \Lambda_{s_{1}^{\ell_1}}(n_1) \times \Lambda_{s_{2}^{\ell_2}}(n_2) \times \cdots \times \Lambda_{s_{k}^{\ell_k}}(n_{k}),
  \end{align}
 such that $s_i \not\equiv \pm s_j \emph{ mod } \mathbb{Z}$ for all $i \not= j$.
\end{thm}
 Accordingly, the study of  blocks of $\mc{O}^{\mf{g}}$ is reduced to the study of blocks of the following three types: (i) ($s=0$) a BGG category $\mc{O}_{n,\mathbb{Z}}$ of the $\mf{q}(n)$-modules of integer weights, see, e.g., \cite{Br2}.
 (ii) ($s \in  \mathbb{Z}+\frac{1}{2}$) a BGG category $\mc{O}_{n,\frac{1}{2}+\mathbb{Z}}$ of the $\mf{q}(n)$-modules of half-integer weights, see, e.g., \cite{CK}, \cite{CKW}.
 (iii) ($s\notin \mathbb{Z}/2$) a BGG category $\mc{O}_{n,s^{\ell}}$ of the $\mf{q}(n)$-modules of  "$\pm s$-weights", see the definition in Section \ref{ParaCateOfqAndEqu}.

 \subsection{}
Let $\mathfrak{gl}(\ell|n-\ell)$ be the general linear Lie superalgebra with the standard Cartan subalgebra $\mathfrak{h}_{\ell |n-\ell}$  for $1\leq \ell \leq n$.  Another main result of the present paper is to establish an equivalence between a block $\mc{F}_{\la}$ of certain maximal parabolic categories $\mc{F}$ for $\mf{q}(n)$ and certain block of atypicality-one  of the finite-dimensional module category $\mc{F}_{\ell|n-\ell}$ for $\mf{gl}(\ell|n-\ell)$, see the definitions in Sections \ref{subsctionFiniteDimRepnOfgl} and  \ref{ParaCateOfqAndEqu}.  Their identical linkage principle (see Lemma \ref{LinkagePrinciple}) is the first piece of evidence to support such an equivalence.

 For a weight $\la \in \mf{h}_{\bar{0}}^*$, or $\la \in \mf{h}_{\ell|n-\ell}^*$, we denote by $\sharp\la$ the atypicality degree of $\la$ (see, e.g. \cite[Definitions 2.29, 2.49]{CW}). According to  \cite[Theorem 2.6]{Ser98} and \cite[Theorem 1.1]{BS12} the blocks $(\mc{F}_{\ell | n-\ell})_{\la}$ (see Section \ref{subsctionFiniteDimRepnOfgl}) for all $\ell,n-\ell$  with the same $\sharp \la$ are equivalent. More precisely, the endomorphism ring of projective generator of $(\mc{F}_{\ell | n-\ell})_{\la}$ is isomorphic to the opposite ring of the diagram algebra $K^{\infty}_{\sharp\la}$ (see, e.g., \cite[Introduction]{BS12}). In particular, $K^{\infty}_{1}$ is the path algebra of the infinite quiver

\begin{center}
\hskip -3cm \setlength{\unitlength}{0.16in}
\begin{picture}(24,3)
\put(7.8,1){\makebox(0,0)[c]{$\bullet$}}
\put(10.2,1){\makebox(0,0)[c]{$\bullet$}}
\put(15.2,1){\makebox(0,0)[c]{$\bullet$}}
%\put(16.9,1){\makebox(0,0)[c]{$\bullet$}}
\put(12.8,1){\makebox(0,0)[c]{$\bullet$}}
\put(9,1.1){\makebox(0,0)[c]{$\longrightarrow$}}
\put(9,0.7){\makebox(0,0)[c]{$\longleftarrow$}}
\put(11.5,1.1){\makebox(0,0)[c]{$\longrightarrow$}}
\put(11.5,0.7){\makebox(0,0)[c]{$\longleftarrow$}}
\put(14,1.1){\makebox(0,0)[c]{$\longrightarrow$}}
\put(14,0.7){\makebox(0,0)[c]{$\longleftarrow$}}

%\put(8.2,1){\line(1,0){1.5}} \put(10.62,1){\line(1,0){0.5}}
%\put(13.0,1){\line(1,0){1.5}} \put(15.08,1){\line(1,0){1.5}}
%\put(6,1){\line(1,0){1.5}}

\put(6.5,0.95){\makebox(0,0)[c]{$\cdots$}}
\put(16.5,0.95){\makebox(0,0)[c]{$\cdots$}}

\put(9,1.8){\makebox(0,0)[c]{\tiny$x_{i-1} $}}
\put(9,0){\makebox(0,0)[c]{\tiny$y_{i-1}$}}
\put(11.5,1.8){\makebox(0,0)[c]{\tiny$x_{i} $}}
\put(11.5,0){\makebox(0,0)[c]{\tiny$y_{i}$}}
\put(14,1.8){\makebox(0,0)[c]{\tiny$x_{i+1} $}}
\put(14,0){\makebox(0,0)[c]{\tiny$y_{i+1}$}}
\end{picture}
\end{center}
modulo the relations $x_i\cdot y_i =y_{i-1}\cdot x_{i-1}$ and $x_i\cdot x_{i+1} = y_{i+1}\cdot y_{i} = 0$ for all $i\in \mathbb{Z}$.
%We shall first explain the motivation for this equivalence in Lemma \ref{LinkagePrinciple}, which establishes a correspondence preserving linkage principle between irreducibles.
We are now in a position to state the following theorem, which provides a Morita equivalence between $\mc{F}_{\la}$ and $(\mc{F}_{1|1})_0$.

\begin{thm} \label{EndOfProjIsoK1}
Let $L(\la) \in \mc{F}$ with $\sharp \la =1$. Then the endomorphism ring of the projective generator of $\mc{F}_{\la}$ is isomorphic to $(K^{\infty}_{1})^{\emph{op}}$.
\end{thm}

\subsection{} The paper is organized as follows. In Section \ref{SectionPre}, we recall definitions of queer Lie superalgebras, general linear Lie superalgebras and their categories of modules.

In Section \ref{SectionEqu}, an approach of reduction similar to  \cite{CMW} is established for queer Lie superalgebras. Equivalences of  blocks via twisting functors and parabolic induction functors are established. In addition, a description of decomposition of  blocks of $\mc{O}$ is given in Theorem \ref{Decomposition}. %Accordingly, if $\mu \in \Lambda_{s^{\ell}}(n)$ then there is $w\in \mathfrak{S}_{\ell} \times \mathfrak{S}_{n-\ell}$ such that $L(\mu)\in $.

In Section \ref{SEctionEquOfMaxPara}, we recall the category of finite-dimensional modules for $\mf{gl}(\ell|n-\ell)$ and  introduce certain maximal parabolic category for $\mf{q}(n)$. A correspondence preserving linkage principles between their irreducibles is established. Finally, we compute the endomorphism ring of projective generator to obtain Theorem \ref{EndOfProjIsoK1}.\\

\noindent{\bf Acknowledgments.}
The author is very grateful to Shun-Jen Cheng for numerous helpful comments and suggestions.

\section{Preliminaries} \label{SectionPre}
\subsection{Lie superalgebras $\mf{gl}$ and $\mf{q}$}
For positive integers $m,n\geq 1$, let $\mathbb{C}^{m|n}$ be the complex superspace of dimension $(m|n)$. Let $\{ v_{\bar{1}}, \ldots , v_{\bar{m}} \}$ be an ordered basis for the even subspace $\mathbb{C}^{m|0}$ and $\{ v_{1}, \ldots , v_{n} \}$ be an ordered basis for the odd subspace $\mathbb{C}^{0|n}$ so that the general linear Lie superalgebra $\mathfrak{gl}(m|n)$ may be realized as $(m+n) \times (m+n)$ complex matrices indexed by $I(m|n):= \{ \bar{1}< \cdots  <\bar{m}<1<\cdots <n  \}$:
\begin{align} \label{glrealization}
 \left( \begin{array}{cc} A & B\\
C & D\\
\end{array} \right),
\end{align}
where $A,B,C$ and $D$ are respectively $m\times m, m\times n, n\times m, n\times n$ matrices.  For $m=n$, the subspace
\begin{align} \label{qnrealization}
 \mf{g}:= \mathfrak{q}(n)=
\left\{ \left( \begin{array}{cc} A & B\\
B & A\\
\end{array} \right) \middle\vert\ A, B: \ \ n\times n \text{ matrices} \right\}
\end{align}
forms a subalgebra of $\mathfrak{gl}(n|n)$ called the queer Lie superalgebra.

 Let $E_{ab}$ be the elementary matrix in $\mathfrak{gl}(m|n)$ with $(a,b)$-entry $1$ and other entries 0, for $a,b \in I(m|n)$. Then $\{e_{ij}, \bar{e}_{ij}|1\leq i,j \leq n\}$ is a linear basis for $\mathfrak{g}$, where $e_{ij}= E_{\bar{i}\bar{j}}+E_{ij}$ and $\bar{e}_{ij}= E_{\bar{i}j}+E_{i\bar{j}}$. Note that the even subalgebra $\mathfrak{g}_{\bar{0}}$ is spanned by $\{e_{ij}|1\leq i,j\leq n\}$, which is isomorphic to the general linear Lie algebra $\mathfrak{gl}(n)$.

Let $\mf{h}_{m|n}$ and $\mf{h}_{m|n}^*$ be respectively the standard Cartan subalgebra of $\mf{gl}(m|n)$ and its dual space, with linear bases $\{ E_{ii} |  i\in I(m|n)\}$  and $\{ \delta_i | i\in I(m|n)\}$ such that $\delta_i(E_{jj}) =\delta_{i,j}$.

Let $\mathfrak{h} = \mathfrak{h}_{\bar{0}}\oplus \mathfrak{h}_{\bar{1}}$ be the standard Cartan subalgebra of $\mathfrak{g}$, with linear bases $\{h_i:= e_{ii}| 1\leq  i \leq n\}$ and $\{\bar{h}_{i}:= \bar{e}_{ii}|1\leq i \leq n\}$ of $\mathfrak{h}_{\bar{0}}$ and $ \mathfrak{h}_{\bar{1}}$, respectively. Let $\{\varepsilon_i| 1\leq i\leq n\}$ be the basis of $\mathfrak{h}_{\bar{0}}^{*}$ dual to $\{h_i|1\leq i\leq n\}$. We define a symmetric bilinear form $( ,) $ on $\mathfrak{h}_{\bar{0}}^{*}$ by $( \varepsilon_i,\varepsilon_j)  = \delta_{ij}$, for $1\leq i,j\leq n$.

 We denote by $\Phi,\Phi_{\bar{0}},\Phi_{\bar{1}}$ the sets of roots, even roots and odd roots of $\mf{g}$, respectively. Let $\Phi^+=\Phi^+_\even\sqcup\Phi^+_\odd$ be the set of positive roots with respect to its standard Borel subalgebra $\mf b=\mf b_\even\oplus\mf b_\odd$, which consists of matrices of the form \eqref{qnrealization} with $A$ and $B$ upper triangular. Denote the set of negative roots by $\Phi^- : = \Phi \setminus \Phi^+$. Ignoring the parity we have $ \Phi_\even=\Phi_\odd = \{\varepsilon_i - \varepsilon_j| 1\leq i,j \leq n\}$ and $\Phi^+ = \{\varepsilon_i- \varepsilon_j| 1\leq i< j \leq n\}$. We denote by $\leq$ the partial order on $\mf{h}_{\bar{0}}^*$ defined by using $\Phi^+$. The Weyl group $W$ of $\mathfrak{g}$ is defined to be the Weyl group of the reductive Lie algebra $\mathfrak{g}_{\bar{0}}$ and hence acts naturally on $\mathfrak{h}_{\bar{0}}^*$ by permutation. We also denote by $s_{\alpha}$ the reflection associated to a root $\alpha\in \Phi^+$. For a given root $\alpha = \varepsilon_i - \varepsilon_j \in \Phi$, let $\bar{\alpha} :=\varepsilon_i + \varepsilon_j $. For each $\la \in \mathfrak{h}_{\bar{0}}^*$, we have the integral root system $\Phi_{\la}: = \{\alpha\in \Phi | ( \lambda,\alpha) \in \mathbb{Z}\}$ and the integral Weyl group $W_{\lambda}$ defined to be the subgroup of $W$ generated by all reflections $s_{\alpha}$, $\alpha\in \Phi_{\lambda}$.

\subsection{Categories of modules} \label{CatOfModule}
Let $V = V_{\overline{0}}\oplus V_{\overline{1}}$ be a superspace. For a given homogenous element $v\in V_{i}$ ($i \in \mathbb{Z}_2$), we let $\overline{v}$= $i$ denote its parity. Let $\Pi$ denote the parity change functor on the category of superspaces. Let $\Pi^0$ be the identity functor. For a $\mathfrak{g}$-module $M$ and $\mu \in \mathfrak{h}_{\bar{0}}^*$, let $M_{\mu}:=\{m\in M| h\cdot m = \mu(h)m, \text{ for } h\in \mathfrak{h}_{\bar{0}}\}$ denote its $\mu$-weight space. If $M$ has a weight space decomposition $M = \oplus_{\mu\in \mathfrak{h}_{\bar{0}}^*}M_{\mu}$, its character is given as usual by $\text{ch}M = \sum_{\mu \in \mathfrak{h}_{\bar{0}}^*}\text{dim}M_{\mu}e^{\mu} $, where $e$ is an indeterminate. In particular, we have the root space decomposition $\mathfrak{g} = \mf{h} \oplus (\oplus_{\alpha \in \Phi} \mathfrak{g}_{\alpha})$ with respect to the adjoint representation of $\mathfrak{g}$. %Let $\mathfrak{b}:= \mathfrak{h} \oplus \left( \oplus_{\alpha \in \Phi^+}\mathfrak{g}_{\alpha}\right)$ denote the standard Borel subalgebra of $\mathfrak{g}$.

 Let $\la = \sum_{i=1}^{n} \la_i \varepsilon_i  \in \mathfrak{h}_{\bar{0}}^*$, and consider the symmetric bilinear form on $\mathfrak{h}_{\bar{1}}^*$ defined by $\langle\cdot,\cdot\rangle_{\la} : = \la([\cdot,\cdot] )$. Let $\ell(\la)$ be the number of $i$'s with $\la_i \neq 0$ and $\delta(\la) = 0$ (resp. $\delta(\la) =1$) if $\ell(\la)$ is even (resp. odd).  Let $1\leq i_1<i_2< \cdots < i_{\ell(\la)}\leq n$  such that $\la_{i_1}, \la_{i_2}, \ldots , \la_{i_{\ell(\la)}}$ are non-zero. Denote by $\lceil \cdot \rceil$ the ceiling function. Then the space  \begin{align} \mathfrak{h}'_{\bar{1}}:=
 \left(\oplus_{j\neq i_1,\ldots ,i_{\ell(\la)}}\mathbb{C}\overline{h}_j\right) \oplus \left( \oplus_{k=1}^{\ell(\la) - \lceil \ell(\la)/2 \rceil} \mathbb{C}(\overline{h}_{i_{2k-1}} + \frac{\sqrt{-\la_{i_{2k-1}}}}{\sqrt{\la_{i_{2k}}}} \overline{h}_{i_{2k}})  \right), \end{align}
 is a maximal isotropic subspace  of $\mathfrak{h}_{\bar{1}}$ associated to $\langle\cdot,\cdot\rangle_{\la} $. Put $\mathfrak{h}' =  \mathfrak{h}_{\bar{0}} \oplus \mathfrak{h}'_{\bar{1}}$. Let $\mathbb{C}v_{\la}$ be the one-dimensional $\mathfrak{h}'$-module with $\overline{v_{\la}} = \bar{0}$, $h\cdot v_{\la} = \la(h)v_{\la}$ and $h' \cdot v_{\la} =0 $ for $h\in \mathfrak{h}_{\bar{0}}$, $h' \in \mathfrak{h}'_{\bar{1}}$. Then $I_{\la} : = \text{Ind}_{\mathfrak{h}'}^{\mathfrak{h}}\mathbb{C}v_{\la}$ is an irreducible $\mathfrak{h}$-module of dimension $2^{\lceil \ell(\la)/2 \rceil}$ (see, e.g., \cite[Section 1.5.4]{CW}). We let $M(\la): = \text{Ind}_{\mathfrak{b}}^{\mathfrak{g}}I_{\la}$ be the Verma module, where $I_{\la}$ is extended to a $\mathfrak{b}$-module in a trivial way, and define $L(\la)$ to be the unique irreducible quotient of $M(\la)$.

Let $\mc{O}^{\mf g}$ denote the BGG category (see, e.g., \cite[Section 3]{Fr}) of finitely generated $\mf{g}$-modules which are locally finite over $\mf{b}$ and semisimple over $\mf{h}_{\bar{0}}$. Note that the morphisms in $\mc{O}^{\mf{g}}$ are even.
%The set $\{L(\la)| \la \in \mf{h}_{\bar{0}}^*\}$ is a complete set of irreducible $\mf{g}$-modules in $\mc{O}^{\mf{g}}$ up to isomorphism and parity-change.
%More precisely, we assign each $\la \in \mf{h}_{\bar{0}}^*$ an irreducible $\mf{h}$-module $J_{\la}$ and then let $L_{\la}$ be the irreducible object of $\mc{O}$ with highest space $J_{\la}$.
It is known (see, e.g., \cite[Section 1.5.4]{CW}) that $L(\la)\cong \Pi L(\la)$ if and only if $\delta(\la) = 1$. Therefore we have the following.
\begin{lem} \label{ClarificationOfParityOfL}
 $\{L(\la)| \la \in \mf{h}_{\bar{0}}^* \text{ with } \delta(\la) = 1 \} \cup \{L(\la), \Pi L(\la)| \la \in \mf{h}_{\bar{0}}^* \text{ with } \delta(\la) = 0 \}$ is a complete set of irreducible $\mf{g}$-modules in $\mc{O}^{\mf{g}}$ up to isomorphism.
\end{lem}
%\begin{proof} Let $\mc{C}$ be the Clifford superalgebra of one variable, then we have (see, e.g. \cite[Section 1.5.4]{CW}) \begin{align} \label{EndOfLforqn} \text{End}_{\mf{g}}(L_{\la}) \cong \left\{ \begin{array}{ll}    \mathbb{C} \text{ if $\delta(\la) =0$  },\\ \mc{C} \text{ if $\delta(\la) =1$ }. \end{array}\right. \end{align} The completes the proof. \end{proof}

 We denote by $Z(\mf{g})$ the {\em center} of U($\mf{g}$). As in the case of Lie algebras, the BGG category $\mathcal{O}^{\mathfrak{g}}$  of $\mathfrak{g}$ has a decomposition into subcategories corresponding to central characters $\chi_{\la}: Z(\mathfrak{g}) \rightarrow \mathbb{C}$ for $\lambda \in \mathfrak{h}_{\bar{0}}^*$. We have a refined decomposition by the linkage principle (see, e.g., \cite[Section 2.3]{CW})
\begin{align} \label{qnBasicLinkagePrinciple}
\mathcal{O}^{\mathfrak{g}} =   \bigoplus_{\lambda\in \mathfrak{h}_{\bar{0}}^* / \sim} \overline{\mathcal{O}}^{\mathfrak{g}}_{\lambda},
\end{align}
where the equivalence relation $\sim$ on $\mathfrak{h}_{\bar{0}}^{*}$ is defined by
 \begin{align} \label{qnLinkagePrinciple}
 \lambda \sim \mu  \text{ if and only if }  \chi_{\la} = \chi_{\mu} \text{ and } \mu \in \la +\mathbb{Z}\Phi,
 \end{align}
 and  $\overline{\mathcal{O}}^{\mathfrak{g}}_{\lambda}$ is the Serre subcategory of $\mathcal{O}^{\mathfrak{g}}$ generated by simple objects with highest weight $\mu$  such that $\la \sim \mu$. The subcategories $\overline{\mathcal{O}}^{\mathfrak{g}}_{\lambda}$ are decomposable in general.

For a finite direct sum of queer Lie superalgebras and reductive Lie algebras, we have analogous notation and decomposition of its BGG category.  When there is no confusion, we denote $\mc{O}^{\mf g}$ by $\mc{O}$. For $\la \in \mf{h}_{\bar{0}}^*$, denote the block of $\mc{O}$ containing $L(\la)$ by $\mc{O}_{\la}$. Namely, it is the Serre subcategory generated by the set of vertices in the connected component of the Ext-quiver for $\mc{O}$ containing $L(\la)$.

 \section{Equivalences and Reductions for Blocks of Queer Lie Superalgebra} \label{SectionEqu}
 \subsection{Equivalence  using twisting functors}

 For a simple root $\alpha \in \Phi^+$, we can defined the {\em twisting functor} $T_{\alpha}$ associated to $\alpha$. The twisting functor was originally defined by Arkhipov in \cite{Ar} and further investigated in more detail in \cite{AS}, \cite{KM}, \cite{CMW}, \cite{AL}, \cite{MS}, \cite{GG13}, \cite{KM}. Recall the precise definition of $T_{\alpha}$ as follows. First, fix a non-zero root vector $X \in (\mathfrak{g}_{\bar{0}})_{-\alpha}$. Since the adjoint action of $X$ on $\mathfrak{g}$ is nilpotent, by using a standard argument (see e.g. \cite[Lemma 4.2]{MO00}) we can form the Ore localization $U'_{\alpha}$ of $U(\mathfrak{g})$ with respect to the set of powers of $X$. Since $X$ is not a zero divisor in $U(\mathfrak{g})$, $U(\mathfrak{g})$ can be viewed as an associative subalgebra of $U'_{\alpha}$. The quotient $U'_{\alpha}/U(\mathfrak{g})$ has the induced structure of a $U(\mathfrak{g})$-$U(\mathfrak{g})$-bimodule. Let $\varphi=\varphi_{\alpha}$ be an automorphism of $\mathfrak{g}$ that maps $(\mathfrak{g}_{i})_{\beta}$ to  $(\mathfrak{g}_{i})_{s_{\alpha}(\beta)}$ for all simple root $\beta$ and $i\in \{\bar{0}, \bar{1}\}$. Finally, consider the bimodule $^{\varphi}U_{\alpha}$, which is obtained from $U_{\alpha}$ by twisting the left action of $U(\mathfrak{g})$ by $\varphi$. We also have an analogous construction with respect to the subalgebra $\mathfrak{g}_{\bar{0}}$ to obtain the $U(\mathfrak{g}_{\bar{0}})$-$U(\mathfrak{g}_{\bar{0}})$-bimodule $^{\varphi}U_{\alpha}^{\bar{0}}$. Now we are in a position to define twisting functors:
 \begin{align}
 T_{\alpha}(-):=  ^{\varphi}U_{\alpha}\otimes - : \mathcal{O}^{\mathfrak{g}} \rightarrow \mathcal{O}^{\mathfrak{g}} \text{ and }
 T_{\alpha}^{\bar{0}}(-):= ^{\varphi}U_{\alpha}^{\bar{0}}\otimes - : \mathcal{O}^{\mathfrak{g}_{\bar{0}}} \rightarrow \mathcal{O}^{\mathfrak{g}_{\bar{0} }}.
 \end{align}
  Then $T_{\alpha}$ and $T_{\alpha}^{\bar{0}}$ have right adjoints $K_{\alpha}$ and $K_{\alpha}^{\bar{0}}$, respectively (see, e.g., \cite{AS}). Let $\mathcal{D}^b(\mathcal{O})$ and $\mathcal{D}^b(\mathcal{O}^{\mf{g}_{\bar{0}}})$ be the bounded derived categories of $\mathcal{O}$ and $\mathcal{O}^{\mf{g}_{\bar{0}}}$, respectively. It is not hard to prove that $T_{\alpha}$ and $T_{\alpha}^{\bar{0}}$ are right exact functors. Let $\mathcal{L}_iT_{\alpha}$, $\mathcal{L}_iT_{\alpha}^{\bar{0}}$ the $i$-th left derived functors of $T_{\alpha}$, $T_{\alpha}^{\bar{0}}$, respectively. It was proved in \cite{AS} that $\mathcal{L}_iT_{\alpha}^{\bar{0}} = 0$ for $i> 1$ and $\mathcal{L}_1T_{\alpha}^{\bar{0}}$ is isomorphic to the functor of taking the maximal submodule on which the action of $\mathfrak{g}_{-\alpha}$ is locally nilpotent.  Similarly, we have analogous definition for  right derived endofunctors $\mathcal{R}^iK_{\alpha}$ and $\mathcal{R}^iK_{\alpha}^{\bar{0}}$ of $K_{\alpha}$ and $K_{\alpha}^{\bar{0}}$, respectively. Furthermore, $\mathcal{R}^iK_{\alpha}^{\bar{0}} = 0$ for $i> 1$ and $\mathcal{R}^1K_{\alpha}^{\bar{0}}$ is isomorphic to the functor of taking the maximal subquotient on which the action of $\mathfrak{g}_{-\alpha}$ is locally nilpotent.

   The star action $\ast$ of $s_{\alpha}$ on weights had been introduced in \cite[Introduction]{GG13} and \cite{CM}: $s_{\alpha}\ast \la := s_{\alpha}\la$ if $( \la, \bar{\alpha}) \not= 0$ and $s_{\alpha}\ast \la := s_{\alpha}\la -\alpha$ if $( \la, \bar{\alpha}) = 0$. We call the former an $\alpha$-typical weight and the later an $\alpha$-atypical weight (also see \cite[Section 1.2.3]{GG13}). The following theorem is inspired by \cite[Proposition 8.6]{CM}.
 \begin{thm} \label{TwistingFunctorthm}
Let $\la \in \mf{h}_{\bar{0}}^*$ and $\alpha \in \Phi^+$ be a simple root such that $( \lambda, \alpha) \notin \mathbb{Z}$.  Then $ \Pi^i \circ T_{\alpha}: \mathcal{O}_{\la} \rightarrow \mathcal{O}_{s_{\alpha}\lambda}$ is an equivalence with inverse $ \Pi^j \circ K_{\alpha}: \mathcal{O}_{s_{\alpha}\la} \rightarrow  \mathcal{O}_{\la}$ for some $i,j \in \{0,1\}$.
 \end{thm}
 \begin{proof}
   We claim that $T_{\alpha}$ and $K_{\alpha}$  are exact functors on $\mathcal{O}_{\la}$ and $\mathcal{O}_{s_{\alpha}\lambda}$, respectively. To see this, we first note that $\mathcal{L}_1T_{\alpha}^{\bar{0}}$ and $\mathcal{R}^1K_{\alpha}^{\bar{0}}$ vanish at each simple $\mathfrak{g}_{\bar{0}}$-module of highest weight $\mu$ with $( \mu,\alpha )\not \in \mathbb{Z}$ (e.g., \cite[Chapter 3]{Mar}). Next we recall that $\text{Res}_{\mathfrak{g}_{\bar{0}}}^{\mathfrak{g}}\circ \mathcal{L}_iT_{\alpha} = \mathcal{L}_iT_{\alpha}^{\bar{0}}   \circ  \text{Res}_{\mathfrak{g}_{\bar{0}}}^{\mathfrak{g}}$  and   $\text{Res}_{\mathfrak{g}_{\bar{0}}}^{\mathfrak{g}}\circ \mathcal{R}^iK_{\alpha} = \mathcal{R}^iK_{\alpha}^{\bar{0}}   \circ  \text{Res}_{\mathfrak{g}_{\bar{0}}}^{\mathfrak{g}}$ (e.g. \cite[Lemma 5.1]{CM}) for all $i\geq 0$. This means that $\mathcal{L}_iT_{\alpha}M = \mathcal{R}^iK_{\alpha}M' =0$ for all $M\in \mathcal{O}_{\la}, M'\in \mathcal{O}_{s_{\alpha}\la}$ and $i\geq 1$. As a conclusion, $T_{\alpha}$ and $K_{\alpha}$ are exact functors on $\mathcal{O}_{\la}$ and $\mathcal{O}_{s_{\alpha}\lambda}$, respectively. For $\mu\in \mf{h}_{\bar{0}}^*$ with  $(\mu, \alpha) \not \in \mathbb{Z}$, it is proved in \cite[Lemma 5.8]{CM} that $T_{\alpha}L(\mu)$ is simple with $T_{\alpha}^2L(\mu) \in \{ L(\mu), \Pi L(\mu)\}$. By a similar argument we can show that $K_{\alpha}L(\mu)$ is also simple with $K_{\alpha}^2L(\mu) \in \{  L(\mu), \Pi L(\mu)\}$. That is, that $T_{\alpha}$ and $K_{\alpha}$ preserve simple objects of $\mathcal{O}_{\la}$ and $\mathcal{O}_{s_{\alpha}\la}$, respectively. Finally recall that $\text{ch}T_{\alpha}M(\mu) = \text{ch}M(s_{\alpha}\mu)$ \cite[Lemma 5.5]{CM} for all $\mu \in \mathfrak{h}_{\bar{0}}^*$. From this together with the fact that $\text{Hom}_{\mathcal{O}}(T_{\alpha}L, L') = \text{Hom}_{\mathcal{O}}(L, K_{\alpha}L') $ for all simple objects  $L, L'\in \mc{O}$, we conclude that $T_{\alpha}$ sends objects of $\mathcal{O}_{\la}$ to objects of $\mathcal{O}_{s_{\alpha}\la}$ and $K_{\alpha}$ sends objects of $\mathcal{O}_{s_{\alpha}\la}$ to objects of $\Pi^i \mathcal{O}_{\la}$, for some $i=0,1$. Consequently, the restrictions of $T_{\alpha}$ and $K_{\alpha}$ make $T_{\alpha}: \mathcal{O}_{\la} \rightarrow \Pi^i \mathcal{O}_{s_{\alpha}\la}$  an equivalence with inverse $K_{\alpha}: \mathcal{O}_{s_{\alpha}\la} \rightarrow  \Pi^j \mathcal{O}_{\la}$, for some $i,j \in \{0,1\}$.
 \end{proof}

 \begin{rem} \label{TwistingFunctorrem}
 Let $\la,\alpha$ be as in Theorem \ref{TwistingFunctorthm}. It is worth pointing out that $T_{\alpha}L(\la)$ only depends on whether $\la$ is $\alpha$-typical or $\alpha$-atypical. That is, it was determined in \cite[Corollary 8.15]{CM}: By the classification of simple $\mathfrak{q}(2)$-highest weight modules in \cite{Mar10} we have $[T_{\alpha}L(\la) : \Pi^i L(s_{\alpha}\ast \la)] \neq 0$, for some $i=0,1$. This is also proved in \cite[Proposition 4.7.1]{GG13}.  As a consequence, we have $T_{\alpha}L(\la) = \Pi^i L(s_{\alpha}\ast \la)$ for some $i=0,1$.
 \end{rem}

 \begin{example}
Let $n=3$ and $\lambda := (-\pi,\pi,-\pi)$. Then by Theorem \ref{TwistingFunctorthm} and Remark \ref{TwistingFunctorrem}, $T_{\epsilon_1 -\epsilon_2} : \mathcal{O}_{\lambda} \rightarrow \mathcal{O}_{\widetilde{\la}}$ is an equivalence sending $L(\la)$ to $\Pi^i L(\widetilde{\la} -(\epsilon_1 -\epsilon_2))$ for some $i=0,1$, where $\widetilde{\la} = (\pi,-\pi,-\pi)$ .
\end{example}

\subsection{Equivalence  using parabolic induction functor}
The goal of this section is to show that the parabolic induction functors give equivalences of blocks under some suitable condition. For given integers $\ell,m$ with $1\leq \ell \leq m $ and $s\in \mathbb{C}$, recall the set $\Lambda_{s^{\ell}}(m)$ defined in \ref{DefOfLambda}. In this section, we consider  blocks  $\mathcal{O}_{\la}$ with the weight $\la \in \mf{h}_{\bar{0}}^*$ of the following form
 \begin{align} \label{AMainConditionOnWeights}
 \lambda \in \Lambda_{s_{1}^{\ell_1}}(n_1) \times \cdots \times \Lambda_{s_{k}^{\ell_k}}(n_{k})
  \text{ such that } s_1=0,\ \ s_2=\frac{1}{2} \text{ and } s_i \not\equiv \pm s_j \emph{ mod } \mathbb{Z},
    \end{align} for all $i \not= j$.
    %In this case, there is a natural Levi subalgebra (see \cite[Section 2.1]{CMW}): $\mathfrak{h}\oplus\left(\oplus_{\alpha \in \Phi_{\la}}\mathfrak{g}_{\alpha}\right)$ and a parabolic subalgebra $\mathfrak{h}\oplus\left(\oplus_{\alpha \in \Phi_{\la}\cup \Phi^+}\mathfrak{g}_{\alpha}\right)$ corresponding to the integral root system $\Phi_{\la}$. In the same paper \cite[Proposition 3.6]{CMW}, the authors show that if we consider the BGG categories of the general linear Lie superalgebra and its Levi subalgebra defined as above, then the parabolic induction functor gives an equivalence between the blocks of $\la$. But in the case of $\mathfrak{q}(n)$, it is not hard to see that the parabolic induction functor defined in this way does not always preserve simple objects and so it is not always an equivalence. The main reason for this phenomena is that if one can obtain an equivalence in this way then the atypicalities prevent narrowing the size of Levi subalgebra (see Lemma \ref{Decompositonlem}). Therefore an alternative reduction is to contain those roots on which the atypicalities occurs. More precisely,
We define $\overline{\Phi}_{\la}:= \{\alpha\in \Phi |  ( \lambda,\bar{\alpha}) \in \mathbb{Z} \}\bigcup \Phi_{\la}$ and $\mathfrak{l}_{\la} := \mathfrak{h}\oplus\left(\oplus_{\alpha \in \overline{\Phi}_{\la}}\mathfrak{g}_{\alpha} \right)$ to be the Levi subalgebra associated to $\la$. In this case, we denote by $\mathfrak{u}_{\la}$ the corresponding nilradical. Furthermore, we have isomorphisms
\begin{align} \label{WeylgroupIso} W_{\la} \cong \mathfrak{S}_{n_1}\times \mathfrak{S}_{n_2} \times (\mathfrak{S}_{\ell_3}\times \mathfrak{S}_{n_3 -\ell_3}) \times \cdots \times (\mathfrak{S}_{\ell_k}\times \mathfrak{S}_{n_k -\ell_k}), \end{align}
and $\mathfrak{l}_{\la} \cong  \mathfrak{q}(n_1) \times \mathfrak{q}(n_2) \times \cdots  \times \mathfrak{q}(n_k)$. In order to prove that the parabolic induction functors are equivalences in this setting, we first recall the following well-known characterization of central characters (see, e.g., \cite[Theorem 2.48]{CW}).

 \begin{lem} \label{centralch}
 For $\lambda , \mu \in \mathfrak{h}_{\bar{0}}^{*}$, $\chi_{\lambda} = \chi_{\mu}$ if and only if there exist $ w \in W$, $\{k_j\}_j \subset \mathbb{C}$, and a subset of mutually orthogonal roots $\{\alpha_j\}_j$ such that $\mu = w(\lambda - \sum_j k_j \alpha_j)$ and $( \lambda, \overline{\alpha_j} ) =0$ for all $j$ .
 \end{lem}

Define a relation $\approx$ on $\mf{h}_{\bar{0}}^*$ as follows. For $\la, \mu \in \mf{h}_{\bar{0}}^*$ we let $\la\approx \mu$ if there exist $ w \in W_{\lambda}$, $\{k_j\}\subset \mathbb{Z}$, and a subset of mutually orthogonal roots $\{\alpha_j\}$ such that $\mu = w(\lambda - \sum_j k_j \alpha_j)$ and $( \lambda, \overline{\alpha_j} ) =0$ for all $j$.  The following lemma shows that $\sim$ and $\approx$ coincide in our setting.

%% The following is a new version
\begin{lem} \label{Decompositonlem}
 %Let $\la \in \mathfrak{h}_{\bar{0}}^{*}$ and suppose that $\la$ is of the form \begin{align} \la \in \Lambda_{s_{1}^{\ell_1}}(n_1) \times \Lambda_{s_{2}^{\ell_1}}(n_2) \times \cdots \times \Lambda_{s_{k}^{\ell_k}}(n_{k}),  \end{align}  such that  $s_1=0, s_2=\frac{1}{2}$ and $s_i \not\equiv s_j \emph{ mod } \mathbb{Z}$ for all $i \not= j$.
 Let $\la \in \mf{h}_{\bar{0}}^*$ be of the form \eqref{AMainConditionOnWeights}. Then $\mu \sim \la$ if and only if $\mu \approx \la$. In particular,
     if $\Pi^i L(\mu) \in \mathcal{O}_{\la}$, for some $i=1,2$, then $\mu \approx \la$.
     %there exist $ w \in W_{\lambda},\{k_j\}\subset \mathbb{Z}$, and a subset of mutually orthogonal roots $\{\alpha_j\}$ such that $\mu = w(\lambda - \sum_j k_j \alpha_j)$ and $\langle \lambda, \overline{\alpha_j} \rangle =0$ for all $j$.
\end{lem}
\begin{proof}
 Since $\chi_{\lambda} = \chi_{\mu}$ we have $\mu = w(\lambda - \sum_jk_j\alpha_j)$ for some $w\in W$, $\{k_j\}_j\subset \mathbb{C}$ and $\{\alpha_j\}_j \subset \Phi$ such that $( \lambda, \overline{\alpha_j} ) =0$ for all $j$   by Lemma \ref{centralch}. Furthermore, we have $\lambda \in \mu +\mathbb{Z}\Phi$. It follows that  $w\in W_{\lambda}$ and $k_j\in \mathbb{Z}$ for all $j$. This completes the proof.
\end{proof}

The following theorem is inspired by \cite[Proposition 3.6]{CMW}.
\begin{thm} \label{ParabolicInductionthm}
 %Let $\la \in \mathfrak{h}_{\bar{0}}^{*}$ and suppose that $\la$ is of the form \begin{align} \la \in \Lambda_{s_{1}^{\ell_1}}(n_1) \times \Lambda_{s_{2}^{\ell_1}}(n_2) \times \cdots \times \Lambda_{s_{k}^{\ell_k}}(n_{k}),   \end{align} such that $s_1=0, s_2=\frac{1}{2}$ and $s_i \not\equiv \pm s_j \emph{ mod } \mathbb{Z}$ for all $i \not= j$.
 Let $\la \in \mf{h}_{\bar{0}}^*$ be of the form \eqref{AMainConditionOnWeights}.
   Let $\mathfrak{l}:= \mathfrak{l}_{\la},  \mathfrak{u}:= \mathfrak{u}_{\la}$. Then there are $i,j \in \{0,1\}$ such that the parabolic induction functor $\Pi^i \circ \emph{Ind}_{\mathfrak{l}+\mathfrak{u}}^{\mathfrak{g}}:\mathcal{O}^{\mathfrak{l}}_{\la}  \rightarrow \mathcal{O}_{\la} $ is an  equivalence, with inverse equivalence $\Pi^j \circ \emph{Res}_{\mathfrak{l}}^{\mathfrak{g}}: \mathcal{O}_{\la} \rightarrow \mathcal{O}^{\mathfrak{l}}_{\la}$ defined by $M \mapsto M^{\mathfrak{u}}$, where $ M^{\mathfrak{u}}$ is the maximal trivial $\mathfrak{u}$-submodule of $M$.
\end{thm}
\begin{proof}
As in the proof of \cite[Propositon 3.6]{CMW}, it suffices to show that $\text{Ind}_{\mathfrak{l}+\mathfrak{u}}^{\mathfrak{g}}L^0_{\mu}$ is irreducible for each irreducible $\mathfrak{l}$-module $L^0_{\mu}\in \mathcal{O}^{\mathfrak{l}}_{\la} $ of highest weight $\mu$. We first assume that $\zeta \in \mathfrak{h}_{\bar{0}}^*$ is a weight of a non-zero singular vector in $\text{Ind}_{\mathfrak{l}+\mathfrak{u}}^{\mathfrak{g}}L^0_{\mu}$. Then by Lemma \ref{Decompositonlem} there exist $ w \in W_{\mu},\{k_j\}_j\subset \mathbb{Z}$, and a subset of mutually orthogonal roots $\{\alpha_j\}_j$ such that $\zeta = w(\mu - \sum_j k_j \alpha_j)$ and $( \mu, \overline{\alpha_j} ) =0$ for all $j$ (note that $\Phi_{\la} = \Phi_{\mu}$). On the other hand, by  consideration of the weights of $\text{Ind}_{\mathfrak{l}+\mathfrak{u}}^{\mathfrak{g}}L^0_{\mu}$, we have $\zeta \in  \mu- \sum_{\alpha\in \Phi^+}\mathbb{Z}_{\geq 0}\alpha$. Hence $  \zeta \in  \mu- \sum_{\alpha\in \overline{\Phi}_{\mu} \cap \Phi^+}\mathbb{Z}_{\geq 0}\alpha$ by \eqref{WeylgroupIso}. This means that every subquotient of $\text{Ind}_{\mathfrak{l}+\mathfrak{u}}^{\mathfrak{g}}L^0_{\mu}$ intersects $L^0_{\mu}$ and so $\text{Ind}_{\mathfrak{l}+\mathfrak{u}}^{\mathfrak{g}}L^0_{\mu}$ is irreducible. This completes the proof.
\end{proof}

\begin{proof}[Proof of Theorem \ref{FirstMainThm}] Let  $\la\in \mathfrak{h}_{\bar{0}}^*$.  We can first apply a sequence of suitable twisting functors (see Theorem \ref{TwistingFunctorthm}) to $\mathcal{O}_{\la}$ and obtain an equivalent block $\mathcal{O}_{\widetilde{\la}}$ such that $ \widetilde{\la}\in \Lambda_{s_{1}^{\ell_1}}(n_1) \times \Lambda_{s_{2}^{\ell_2}}(n_2) \times \cdots \times \Lambda_{s_{k}^{\ell_k}}(n_{k})$ and $s_i \not\equiv \pm s_j \emph{ mod } \mathbb{Z}$ for all $i \not= j$. Next we can apply the parabolic induction functor (see Theorem \ref{ParabolicInductionthm}) to obtain an equivalent block of the desired Levi subalgebra. This completes the proof.
\end{proof}

\begin{example}
Let $\lambda : = (\frac{1}{5},1,-\pi,\frac{3}{2},\pi,-\frac{3}{2},-\pi)$. Then by applying a sequence of twisting functors $\Pi^{i_\alpha} \circ T_{\alpha}$ with some $i_{\alpha} \in \{0,1\}$ in Theorem \ref{TwistingFunctorthm} related to $\alpha$-typical weights, we may transform  $\la$ to the weight $\widetilde{\la} = (1,\frac{1}{5}, \frac{3}{2}, -\frac{3}{2},-\pi,\pi,-\pi)$, which gives an equivalence from $\mathcal{O}_{\la}$ to $\mathcal{O}_{\widetilde{\la}}$ sending $L(\la)$ to $L(\widetilde{\la})$. Then we apply the twisting functor $\Pi^{i}\circ T_{\varepsilon_5 - \varepsilon_6}$ with some $i=0,1$ to obtain the weight $\widetilde{\widetilde{\la}}= (1, \frac{3}{2}, -\frac{3}{2}, \frac{1}{5},\pi,-\pi,-\pi)$ and an equivalence $\mathcal{O}_{\la}$ to $\mathcal{O}_{\widetilde{\widetilde{\la}}}$ which sends $L(\la)$ to $L(\widetilde{\widetilde{\la}} - (\varepsilon_5 - \varepsilon_6))$. Next we use the parabolic induction functors. Define $\alpha : =(\varepsilon_5 - \varepsilon_6)$. Note that
$\widetilde{\widetilde{\la}}, \ \ {\widetilde{\widetilde{\la}}-\alpha} \in \Lambda_{0^{0}}(1) \times \Lambda_{\frac{1}{2}^{1}}(2) \times \Lambda_{\frac{1}{5}^{0}}(1) \times \Lambda_{\pi^{1}}(3)$ and $\mathfrak{l}_{\widetilde{\widetilde{\la}}-\alpha} =\mathfrak{l}_{\widetilde{\widetilde{\la}}} \cong \mathfrak{q}(1) \times \mathfrak{q}(2) \times  \mathfrak{q}(1)\times\mathfrak{q}(3) $. By Theorem \ref{ParabolicInductionthm}, there is an equivalence from $\mathcal{O}_{\la}$ to $\mathcal{O}_{\widetilde{\widetilde{\la}}}^{\mathfrak{l}}$ sending $L(\la)$ to the irreducible $\mathfrak{l}$-module with highest weight $\widetilde{\widetilde{\la}} -\alpha$.
\end{example}

\subsection{Description of  blocks}
\begin{thm} \label{Decomposition}
Let $\la, \mu \in \mf{h}_{\bar{0}}^*$. Then $\Pi^i L(\mu) \in \mc{O}_{\la}$ for some $i = 0,1$ if and only if $\mu \approx \la$.
\end{thm}
\begin{proof}
First assume that $\la \in \mf{h}_{\bar{0}}^*$ is of the form \eqref{AMainConditionOnWeights}. Thanks to Lemma \ref{Decompositonlem}, it remains to show that $\mu \approx \la$ implies $\Pi^i L(\mu) \in \mc{O}_{\la}$ for some $i=0,1$. Recall the fundamental lemma in \cite[Proposition 2.1]{PS2} by Penkov and Serganova. It follows from $\text{Hom}_{\G}(M(\la-\alpha), \Pi^j M(\la))\not=0$  for some $j=0,1$, for all $\alpha\in\Phi^+$ with $ (\la,\ov{\alpha})=0$ that $\Pi^i L(\la-\alpha) \in \mc{O}_{\la}$ for some $i=0,1$. Therefore we may assume that $\mu$ is of the form $s(\la)$, for some reflection $s\in W_{\la}$ corresponding to a simple root $\varepsilon_i-\varepsilon_{i+1}$. In this case, we have $\la_i - \la_{i+1} = k \in \mathbb{Z}$. Without loss of generality, assume that $k>0$. Let $v_{\la}\in M(\la)$ be a highest weight vector, it is not hard to compute that $\overline{e}_{i,i+1} e_{i+1,i}^{k+1}v_{\la}$ is a singular vector in $M(\la)$ of weight $s(\la)$ (see, e.g., \cite[Lemma 2.39]{CW}). This means that $\Pi^i L(\mu) \in \mc{O}_{\la}$ for some $i=0,1$. For arbitrary $\la \in \mf{h}_{\bar{0}}^*$, there are $\la' \in \mf{h}_{\bar{0}}^*$ of the form \eqref{AMainConditionOnWeights} and $T: \mc{O}_{\la} \rightarrow \mc{O}_{\la'}$ an equivalence constructed by using a sequence of twisting functors in Theorem \ref{TwistingFunctorthm}. For $\zeta, \zeta' \in \mf{h}_{\bar{0}}^*$ and simple reflection $s\in W$, note that $s\ast \zeta \approx s\ast \zeta'$ if and only if $\zeta \approx \zeta'$. The theorem now follows by Remark \ref{TwistingFunctorrem}.
\end{proof}

\begin{rem} If $\ell(\la)$ is odd, then $\mc{O}_{\la}$ is the Serre subcategory generated by $\{L(\mu)|\mu \approx \la\}$.
\end{rem}

\begin{comment}
Let $\widetilde{\mc{O}}$ be the BGG category with all $\mf{g}$-morphisms. That is, the set of objects of $\widetilde{\mc{O}}$ consists of objects of $\mc{O}$ but we now allow arbitrary $\mf{g}$-morphisms. Let $\la \in \mf{h}_{\bar{0}}^*$ and define $\widetilde{\mc{O}}_{\la}$ to be the full subcategory generated by the vertices in the component of the Ext-Quiver for $\widetilde{\mc{O}}$ containing $L(\la)$.
\begin{conj}
$\widetilde{\mc{O}}_{\la}$ is the full subcategory of $\widetilde{\mc{O}}$ generated by simple objects $\{L(\mu)| \mu \in \mf{h}_{\bar{0}}^*, \ \ \mu\approx \la\}$.
\end{conj}
\end{comment}

\section{Equivalences of Certain Maximal Parabolic Subcategory} \label{SEctionEquOfMaxPara}
%\begin{thm} Let $\mathfrak{p} = \mathfrak{l} + \mathfrak{u}$ be a maximal parabolic subalgebra of $\mathfrak{g}$. Then the maximal parabolic subcategory $\mathcal{O}^{\mathfrak{p}}$ of $\mathcal{O}$ is equivalent to a certain category of $\mathfrak{gl}(n) \otimes \mathcal{C}$-modules, where $\mathcal{C}$ is the Clifford superalgebra of one variable.
%\end{thm}

In this section, we fix non-negative integers $ n,\ell$ with $n \geq \ell$ and $s\not \in \mathbb{Z}/2$. The goal of this section is to establish an equivalence between certain block of atypicality-one of finite-dimensional category for $\mf{gl}(\ell|n-\ell)$ and some block of certain maximal parabolic subcategory for $\mf{q}(n)$.
\subsection{Finite-dimensional representations of $\mf{gl}(\ell|n-\ell)$} \label{subsctionFiniteDimRepnOfgl}
We denote by $\widetilde{\mc{F}}_{\ell|n-\ell}$   the category of integral weight, finite-dimensional $\mf{gl}(\ell|n-\ell)$-modules with even morphisms. Let $\Lambda^{\mf{a}}: = \oplus_{i=1}^n\mathbb{Z}\delta_i$ be the weight lattice. Recall that the set of all irreducible objects (up to parity) of $\widetilde{\mc{F}}_{\ell|n-\ell}$ are parametrized by its highest weight $\la$ in $ \Lambda^{\mf{a},+} :=\{ \lambda \in \Lambda^{\mf{a}} |  \ \   \lambda_i \geq \lambda_{i+1}  ,\text{ for } 1\leq i< \ell \text{ and } \ell \leq i<n\}$. We define $|\la| : = (\la , \sum_{i=\ell+1}^{n} \delta_i ) \text{ (mod $2$)} $. Recall that for a given $M \in \widetilde{\mc{F}}_{\ell|n-\ell}$, there is a decomposition $M = M_{+} \oplus M_{-}$ of $\mf{gl}(\ell|n-\ell)$-modules, where $M_{+}:=\oplus_{\mu \in \Lambda^{\mf{a}}}(M_{\mu})_{|\mu|}$ and $M_{-}:=\oplus_{\mu \in \Lambda^{\mf{a}}}(M_{\mu})_{|\mu|+1}$. This induces a  decomposition  $\widetilde{\mc{F}}_{\ell|n-\ell} = \mc{F}_{\ell|n-\ell} \oplus \Pi \mc{F}_{\ell|n-\ell}$ (see, e.g., \cite[Section 4-e]{Br1}), where $\mc{F}_{\ell|n-\ell}$ (resp. $\Pi \mc{F}_{\ell|n-\ell}$) is the full subcategory consisting of all $M\in \widetilde{\mc{F}}_{\ell|n-\ell}$ such that $M = M_+$ (resp. $M=M_-$).
 For $\zeta \in \Lambda^{\mf{a},+}$, denote by $(\mc{F}_{{\ell | n-\ell}})_{\zeta}$ the block of $\mc{F}_{{\ell | n-\ell}}$ containing the (unique) irreducible module $L^{\mf{a}}_{\zeta}$ of highest weight $\zeta$. Namely, it is the Serre subcategory generated by the set of vertices in the connected component of the Ext-quiver for $\mc{F}_{\ell | n-\ell}$ containing $L^{\mf{a}}_{\zeta}$.

As we mentioned in Section \ref{Introduction}, the diagram algebra $K^{\infty}_{1}$ is the path algebra of a certain infinite quiver.
Therefore we can identify $(K^{\infty}_{1})^{\text{op}}$ as the associative algebra generated by elements $\{z_i,x_j,y_k\}_{i,j,k \in \mathbb{Z}}$ and relations
\begin{align*}
z_ic = cz_i = y_iy_j =x_jx_i =0,\\
 x_iy_i =z_{i+1},\ \ y_ix_i =z_i,
\end{align*}
$\text{ for all } i,j\in \mathbb{Z}, c\in \{x_s,y_t\}_{s,t\in \mathbb{Z}}$.

%if $\la\in \mf{h}_{\ell|n-\ell}^*$ and $\la' \in \mf{h}_{\ell'|n'-\ell'}^*$ satisfies $\sharp\la =\sharp \la'$ then $(\mc{F}_{\mf{\ell | m}})_{\text{wt}(\la)}$ and $ (\mc{F}_{\mf{\ell ' | m'}})_{\text{wt}(\la')}$ are equivalent.

\begin{comment}
Let $P\in \mc{F}_{\ell|n-\ell}$ be the projective generator of $\mc{F}_{\ell|n-\ell}$. Recall that $\text{End}_{\mc{F}_{\ell|n-\ell}}(P)$ is isomorphic to the diagram algebra $K^{\infty}_{r}$ when. Define an associative algebra $\mf{P}$ over $\mathbb{C}$ by the following generators and their relations:
\begin{align}
\mf{A} : = \mathbb{C}\{z_i,x_j,y_k\}_{i,j,k \in \mathbb{Z}} / \sim_{\mf{P}},
\end{align}
where the relation $\sim_{\mf{P}}$ is given by
\begin{align*}
z_ic = cz_i = y_iy_j =x_jx_i =0 \text{ and } x_iy_i =z_{i+1}, y_ix_i =z_i, \text{ for all } i,j\in \mathbb{Z}, c\in \{x_s,y_t\}_{s,t\in \mathbb{Z}}.
\end{align*}
Namely, $\mf{A}$
\end{comment}

\subsection{Parabolic categories of $\mf{q}(n)$ and Equivalences} \label{ParaCateOfqAndEqu}
 We define a bijection $\cdot^{\mf{a}} : \Lambda_{s^{\ell}}(n) \rightarrow \Lambda^{\mf{a}}$  by
\begin{align}
\la = \sum_{i=1}^{n} \la_i\varepsilon_i \in \Lambda_{s^{\ell}}(n) \longmapsto  \la^{\mf{a}} :=\sum_{i=1}^{\ell} (\la_i -s )\delta_i + \sum_{i=\ell+1}^{n}(\la_i+s)\delta_i -\rho \in   \Lambda^{\mf{a}},
\end{align} where $\rho := \sum_{i=1}^{\ell} -(\ell-i+1)\delta_i + \sum_{i=\ell+1}^{n} (i-\ell)\delta_i$.

   Let $\chi_{\la}^{\mf{a}}$ be the central character of $\mf{gl}(\ell|n-\ell)$ corresponding to $\la \in \mathfrak{h}_{\ell|n-\ell}^*$. We first consider the linkage principles under this bijection.

   Denote by $\Psi:=\{ \delta_i - \delta_j | 1\leq i\neq j\leq n\}$ the root system of $\mf{gl}(\ell|n-\ell)$. Recall that the linkage principle in \eqref{qnLinkagePrinciple} for $\mf{q}(n)$ defines an equivalence relation $\sim$.  The following lemma follows from Lemma \ref{Decompositonlem} and the proof of \cite[Proposition 3.3]{CMW}.
 \begin{lem} \label{LinkagePrinciple}
 Let $\la,\mu \in \Lambda_{s^{\ell}}(n)$. Then $\la \sim \mu$ if and only if $ \chi_{\la^\mf{a}}^{\mf{a}} = \chi_{\mu^\mf{a}}^{\mf{a}} \text{ and } \mu^\mf{a} \in \la^\mf{a} +\mathbb{Z}\Psi$.
 \end{lem}
%\begin{proof} Note that $(\mu -\la)^a = \mu^a -\la^a$. Therefore $\mu \in \la+ \mathbb{Z}\Phi$  if and only if $\mu^a \in \la^a +\mathbb{Z}\Psi$.  \end{proof}
 %It is not hard to show that the linkage principles (see, e.g., \cite[Section 2.2.6, 2.3.3]{CW}) for $ \Lambda_{s^{\ell}}(\ell+m)$ over $\mf{q}(\ell+m)$ and for $ \Lambda_{s^{\ell}}(\ell+m)'$ over $\mf{gl}(\ell|m)$ coincide under this correspondence. That is, weights $\la,\mu \in \ \Lambda_{s^{\ell}}(\ell+m)$ are linked if and only if $\la',\mu'  $ are linked.

We define the set $\Lambda_{s^{\ell}}^+(n) := \{ \lambda \in \Lambda_{s^{\ell}}(n) | \ \  \lambda_i > \lambda_{i+1}  ,\text{ for } 1\leq i< \ell \text{ and } \ell \leq i< n\}$. Note that we have $ \Lambda^{\mf{a},+}= (\Lambda_{s^{\ell}}^+(n))^{\mf{a}}$. For arbitrary $\la \in \Lambda_{s^{\ell}}(n)$ we have a Levi subalgebra $\mf{l}:=  \mathfrak{h}\oplus\left(\oplus_{\alpha \in \Phi_{\la}}\mathfrak{g}_{\alpha}\right) \cong \mf{q}(\ell) \times \mf{q}(n-\ell)$ and the maximal parabolic subalgebras $\mf{p} := \mathfrak{h}\oplus\left(\oplus_{\alpha \in \Phi_{\la}\cup \Phi^+}\mathfrak{g}_{\alpha}\right)$. Let $\mf{u}$ be the corresponding nilradical of $\mf{p}$.  We denote by $\mc{O}^{\mf{p}}$ the maximal parabolic subcategory of (see, e.g., \cite[Section 3.1]{Mar14}) $\mc{O}$ with respect to $\mf{p}$. Namely, $\mc{O}^{\mf{p}}$  is the Serre subcategory of $\mc{O}$ generated by $\mf{p}$-locally finite, and $\mf{l}_{\bar{0}}$-semisimple $\mf{g}$-modules. We define $\mc{O}_{n,s^{\ell}}$ to be the full subcategory of $\mf{g}$-modules in $\mc{O}$ with weights in $\Lambda_{s^{\ell}}(n)$ and $\mc{F}:= \mc{O}^{\mf{p}} \cap \mc{O}_{n,s^{\ell}}$ its maximal parabolic subcategory. For each $M \in \mc{F}$, note that $M$ is also $\mf{l}$-semisimple since all weights of $M$ are $\mf{l}$-typical. As a conclusion, if $\Pi^i L(\mu)\in \mc{F}$ for some $i=0,1$ then we have $\mu \in \Lambda_{s^{\ell}}^+(n)$.

Let $\la \in \Lambda^+_{s^{\ell}}(n)$.   Note that every irreducible $\mf{l}$-module of highest weight $\la$ can be extended to a $\mf p$-module by letting $\mf u$ act trivially.  We define $L^0(\la)$ to be the finite-dimensional irreducible $\mf{l}$-module with highest weight space $I_{\la}$. Therefore the corresponding parabolic Verma module $K(\la):=\text{Ind}_{\mf p}^\G L^0(\la)$ has the irreducible quotient $L(\la)$. Furthermore, we note that $K(\la)$ is $\mf{p}$-locally finite and all the $\mf{l}$-weights of $K(\la)$ are $\mf{l}$-typical. Therefore we have $K(\la),L(\la) \in \mc{F}$.  Consequently, $\mc{F}$ is the Serre subcategory of $\mc{O}$ generated by $\{\Pi^iL(\la)| \la \in \Lambda^+_{s^{\ell}}(n), \ i\in \{0,1\} \}$.

 For $\la \in \Lambda^+_{s^{\ell}}(n)$. We also denote by $P(\la)$ and $U(\la)$ the projective cover of $L(\la)$ and the tilting module corresponding to $\la$ in $\mc{O}^{\mf{p}}$, respectively. For their definitions and existences, we refer to \cite[Proposition 1,7]{Mar14} and \cite[Theorem 2]{Mar14}. Note that all weights of $P(\la),U(\la)$ are in $\Lambda_{s^{\ell}}(n)$ since they are indecomposable (by definition). That is, $P(\la),U(\la) \in \mc{F}$.

  %Note that for a simple object $L\in \mc{O}^{\mf{p}}$, there is possibly an odd automorphism (see, e.g. \cite[Section 1.5.4]{CW})  on $L$ and so $\text{End}_{\mc{O}^{\mf{p}}}(L)$ is isomorphic to the Clifford superalgebra of one variable. But $\text{End}_{\mc{F}_{\ell|m}}(L') \cong \mathbb{C}$ for all simple object $L' \in \mc{F}_{\ell|m}$. Therefore we consider the underlying even category of $\mc{O}^{\mf{p}}$, that is, objects of $\mc{F}$ are objects of $\mc{O}^{\mf{p}}$, but morphisms of $\mc{F}$ are now even morphisms of $\mc{O}^{\mf{p}}$.
  Let $P$ be the free abelian group on basis $\{\epsilon_a\}_{a\in \mathbb{Z}}$. Let $\text{wt}(\cdot) : \Lambda^+_{s^{\ell}}(n) \rightarrow P$ be the weight function defined by (c.f. \cite[Section 2-c]{Br2})
\begin{align}
\text{wt}(\la) := \sum_{i=1}^{\ell}\epsilon_{\la_{i}-s}  + \sum_{i=\ell+1}^{n}(-\epsilon_{-(\la_{i}+s)}).\end{align}
  By Lemma \ref{centralch}, we have $\chi_{\la} = \chi_{\mu}$ if and only if $\text{wt}(\la) = \text{wt}(\mu)$. By  \eqref{qnBasicLinkagePrinciple}, we have  decomposition $\mc F = \oplus_{\la\in \mf{h}_{\bar{0}}^*} \mc{F}_{\chi_{\la}} = \oplus_{\gamma \in P} \mc{F}_{\gamma}$  according to central characters $\chi_{\la}$ with $\text{wt}(\la) = \gamma$.

 Let $\mathbb{C}^{n|n}$ and $(\mathbb{C}^{n|n})^*$  be the standard representation and its dual, respectively. Denote the projection functor from $\mc{F}$ to $\mc{F}_{\gamma}$ by $\text{pr}_{\gamma}$.  We define the {\em translation functors} $\text{E}_{a}, \text{F}_{a}: \mc{F} \rightarrow \mc{F}$ as follows
\begin{align}
\text{E}_{a} (M):= \text{pr}_{\gamma+(\epsilon_a-\epsilon_{a+1})}(M\otimes (\mathbb{C}^{n|n})^*), \ \
\text{F}_{a} (M):= \text{pr}_{\gamma-(\epsilon_a-\epsilon_{a+1})}(M\otimes \mathbb{C}^{n|n}),
\end{align}
for all  $M\in  \mc{F}_{\gamma}$, $\gamma \in P$ , $a\in \mathbb{Z}$. For each $a\in \mathbb{Z}$ , it is not hard to see that both $\text{E}_{a}$ and $\text{F}_{a}$ are exact and bi-adjoint to each other. %Let $\mc{K}(\mc{F})$ be the Grothedieck group of $\mc{F}$ and denote the element corresponding to $M\in \mc{F}$ by $[M]$.
%By tensor identities we may conclude that $
%\text{E}_{a} K(\la)$ and $
%\text{F}_{a} K(\la)$ have filtrations whose non-zero subquotients are parabolic Verma modules.
We write $\la \rightarrow_{a} \mu$ if $\la , \mu \in \Lambda^+_{s^{\ell}}(n)$ and there exists $1\leq i \leq \ell$ such that $\la_i  = \mu_i-1 = a+s$  or there exists $\ell+1 \leq i\leq n$ such that $\la_i = \mu_i-1 = -a -1 -s $, and in addition, $\la_j = \mu _j$ for all $i\neq j$.
We have the following lemma.
\begin{lem} \label{chOfKInTranslationFunctor}
Let $\la\in \Lambda^+_{s^{\ell}}(n)$. Then both $\emph{E}_{a} K(\la)$ and $\emph{F}_{a} K(\la)$ have flags of parabolic Verma modules and  we have the following formula:
\begin{align}
\emph{ch}\emph{E}_{a} K(\la) = 2\sum_{\mu \rightarrow_{a} \la} \emph{ch}K(\mu),\label{ComputationOfTranslationFunctors_1} \\
\emph{ch}\emph{F}_{a} K(\la) = 2\sum_{\la \rightarrow_{a} \mu} \emph{ch}K(\mu). \label{ComputationOfTranslationFunctors_2}
\end{align}
\end{lem}
\begin{proof}
We first note that  we have
\begin{align} \label{ComputationOfTranslationFunctors}
 K(\la) \otimes \mathbb{C}^{n|n} \cong U(\mf{g}) \otimes_\mf{p} (L^0(\la)\otimes \mathbb{C}^{n|n}).  \end{align}
This implies that $K(\la) \otimes \mathbb{C}^{n|n}$ (and so the summand of $\text{F}_{a} K(\la)$) has a filtration whose non-zero subquotients are parabolic Verma modules. Note the $\mf{l}_{\bar{0}} \cong \mf{gl}(\ell)\times \mf{gl}(n-\ell)$ with the Cartan subalgebra $\mf{h}_{\bar{0}} = \oplus_{i=1}^{n}\mathbb{C}e_{ii}$ and its dual $\mf{h}_{\bar{0}}^* = \oplus_{i=1}^{n}\mathbb{C}\varepsilon_i$. Let $^{(1)}\zeta:= \sum_{i=1}^{\ell}\la_i\varepsilon_i$ and $\zeta^{(1)}:=\sum_{i=\ell+1}^n \la_i \varepsilon_i$,  and $\rho^{\ell}_{\bar{0}},\rho^{n-\ell}_{\bar{0}}$ be Weyl vectors for the general linear Lie algebras $\mf{gl}(\ell)$ and $ \mf{gl}(n-\ell)$, respectively. Then the module in  \eqref{ComputationOfTranslationFunctors} has character
\begin{align}
\text{ch} \left( K(\la) \otimes \mathbb{C}^{n|n} \right) =2\cdot D \cdot \sum_{i=1}^n \left( e^{\varepsilon_{i}} \cdot s_{^{(1)}\la - \rho^{\ell}_{\bar{0}}}\cdot  s_{\la^{(1)} - \rho^{n-\ell}_{\bar{0}}} \right).
\end{align}
where $s_{^{(1)}\la - \rho^{\ell}_{\bar{0}}}$ and $ s_{\la^{(1)} - \rho^{n-\ell}_{\bar{0}}}$ are respectively Schur functions corresponding to $^{(1)}\la - \rho^{\ell}_{\bar{0}}$ and $\la^{(1)}- \rho^{n-\ell}_{\bar{0}}$, respectively, and \begin{align} D = 2^{\lceil n/2\rceil} \prod_{1\leq i< j \leq \ell} \frac{1+e^{-\varepsilon_i+\varepsilon_j}}{1-e^{-\varepsilon_i+\varepsilon_j}} \cdot \prod_{\ell+ 1\leq i< j \leq n} \frac{1+e^{-\varepsilon_i+\varepsilon_j}}{1-e^{-\varepsilon_i+\varepsilon_j}}, \end{align} \begin{align} D\cdot {s_{^{(1)}\zeta - \rho^{\ell}_{\bar{0}}}\cdot s_{\zeta^{(1)} - \rho^{n-\ell}_{\bar{0}}}} = {\text{ch}K(\zeta)}, \end{align}
for all $\zeta \in \Lambda^+_{s^{\ell}}(n)$ (c.f. \cite[Theorem 2]{Pe} and \cite[Section 3.1.3]{CW}).
By the Pieri formula (see, e.g., \cite[Lemma 5.16]{Mac95}), we may conclude that $\text{ch} \left( K(\la) \otimes \mathbb{C}^{n|n} \right) = 2\sum_{\mu} \text{ch} K(\mu)$, where the summation is over $\mu$ such that  $\mu = \la +\varepsilon_i$, for some $1 \leq i\leq n$, and $\mu_1 > \cdots >\mu_{\ell}$, $\mu_{\ell+1} > \cdots >\mu_n$. From the definition of $\text{F}_a$, we obtain \eqref{ComputationOfTranslationFunctors_2}.

Formula \eqref{ComputationOfTranslationFunctors_1} can be obtained by a similar argument.
\end{proof}

Let $\la\in \Lambda^+_{s^{\ell}}(n)$. It is not hard to prove that both $\text{E}_aU(\la)$ and $\text{F}_aU(\la)$ are direct sums of tilting modules (see, e.g., \cite[Corollary 4.27]{Br1}). Furthermore, we have the following lemma.
\begin{lem} \label{TransFunOfTiltings}
Let $\la\in \Lambda^+_{s^{\ell}}(n)$. Then the multiplicity of each non-zero tilting summand of $\emph{E}_aU(\la)$ and $\emph{F}_aU(\la)$ is even.
\end{lem}
\begin{proof}
We first let $U(\nu_1),\ldots ,U(\nu_p)$ be the (not necessarily distinct) direct summands of $\text{F}_aU(\la)$. Suppose on the contrary that $[\text{F}_aU(\la): U(\nu_i)] $ is odd for some $1\leq i\leq p$. Without loss of generality, we let  $1\leq p'\leq p$ such that $[\text{F}_aU(\la) :U(\nu_i)]$ are odd for all $i\leq p'$, and $[\text{F}_aU(\la) :U(\nu_{i'})]$ are even for all $i'>p'$, and in addition, $\nu_1$ is a maximal element in $\{\nu_i | 1\leq i \leq p'\}$.  Since the coefficients of $\{\text{ch}K(\mu)| \mu\in \mf{h}_{\bar{0}}^*\}$ in $\text{ch}\text{F}_aU(\la)$ are even by Lemma \ref{chOfKInTranslationFunctor}, we have
\begin{align} \label{TranslationFunOfTilting}
 \sum_{i=1}^{p'} \text{ch}U(\nu_i)= \text{ch}\text{F}_aU(\la) - \sum_{i=p'+1}^{p} \text{ch}U(\nu_i) = 2\sum_{i=1}^{q}  \text{ch}K(\mu_i),
\end{align}
 for some (not necessarily distinct) $\mu_1,\mu_2,\ldots,\mu_q \in \mf{h}_{\bar{0}}^*$.
For a given $\nu \in \mf{h}_{\bar{0}}^*$, note that  $\text{ch}U(\nu) \in \text{ch}K(\nu) + \bigoplus_{\zeta< \nu }\mathbb{Z}_{\geq 0} \text{ch}K(\zeta)$. This means that the coefficient of $\text{ch}K(\nu_1)$ in each $\text{ch}U(\nu_j)$'s is zero for all $1\leq j\leq p'$ with $\nu_j \neq \nu_1$. But $[\text{F}_aU(\la): U(\nu_1)]= |\{ j| 1\leq j\leq p',\ \ \nu_j =\nu_1  \}|$ is the coefficient of $\text{ch}K(\nu_1)$ in $ \sum_{i=1}^{p'} \text{ch}U(\nu_i)$  contradicting \eqref{TranslationFunOfTilting}.

The result for $E_{a}U(\la)$ can be obtained by a similar argument. This completes the proof.
\end{proof}

%% old description of $\lambda^-$:
\begin{comment}
Let $\la \in \Lambda^+_{s^{\ell}}(n)$ with $\sharp{\la} =1$. We let $1\leq  p\leq p' \leq \ell$ and $\ell< q' \leq q \leq n$ such that $ \la_p  + \la_q =0$, $\{-\la_{p'+1} , \la_{q'-1}\} \cap \{-\la_{p'}+1 , \la_{q'}+1\} =  \phi$ and the numbers $-(\la_p -s), \ \ \ldots ,\ \ -(\la_{p'} -s), \ \ \la_{q'} +s,\ \ \ldots, \ \  \la_{q}+s $ form a sequence of consecutive integers after a rearrangement if needed. Define $\la^- \in \Lambda^+_{s^{\ell}}(n)$ as follows.
\begin{align}
(\la^-)_i := \left\{ \begin{array}{llllll}  \la_{i} -1 \text{ if } i=p ,
\\  \la_{i+1} \text{ if } p <  i <p' ,
\\ -\text{max}\{-(\la_{p'}-s), \la_{q'}+s\}+s-1 \text{ if } i=p' ,
\\  \la_{i} +1 \text{ if } i=q ,
\\ \la_{i-1} \text{ if } q' <  i <q ,
\\ \text{max}\{-(\la_{p'}-s), \la_{q'}+s\} -s +1 \text{ if } i=q' .
 \end{array} \right.
 \end{align}
 \end{comment}

\begin{comment} \begin{align}
\la_-: = \sum_{i=1}^{p-1} \la_i\varepsilon_i  + (\la_{p}-1)\varepsilon_p  + \sum_{i=p+1}^{p'-1} \la_{i+1}\varepsilon_i + (\la_{p'}-1)\varepsilon_{p'+1}+ \sum_{i=p'+1}^{\ell}\la_i\varepsilon_i   +\sum_{i=1}^{p-1} .
\end{align}
\end{comment}
By Lemma \ref{TransFunOfTiltings} and algorithm in \cite[Procedure 3.20]{Br1}, we have the following lemma:
%Since actions of translation functors on Grothendieck groups coincide (up to $2$-factor) under both setting, we have the following lemma.
\begin{lem} \label{ProCovFromBrsAlgorithm}
Let $\la \in \Lambda^+_{s^{\ell}}(n)$ with $\sharp{\la} =1$. There is an operation $Z_{\la}$ consisting of first applying the operators $X_a$'s coming from  Brundan's procedure \cite[Procedure 3.20]{Br1} and then taking a direct summand of a direct sum of two isomorphic copies such that
\begin{align}
\emph{ch}Z_{\la}U(t_{\la}) = \emph{ch}K(\la)+ \emph{ch}K(\la^-),
\end{align}
where $t_{\la} \in \Lambda^+_{s^{\ell}}(n)$ is typical, and $\lambda^- = w(\la -k\alpha)$, where $\alpha \in \Phi^+$ with $(\la,\overline{\alpha})=0$ and $k\in \mathbb{N}$ is the smallest positive integer such that $\la -k\alpha$ is $W$-conjugate to an element in $\Lambda^+_{s^{\ell}}(n)$, and $w\in W$ is such that $w(\la-k\alpha)\in\Lambda^+_{s^{\ell}}(n)$.

\begin{comment}
Furthermore, we have $(\la^+)^- = (\la^-)^+ =\la$ and $(-\omega_0\la)^{\pm} = -\omega_0 (\la^{\mp})$.
\end{comment}
\end{lem}
\begin{rem}
Note that the map $\la \in \Lambda^+_{s^{\ell}}(n) \mapsto \la^- \in \Lambda^+_{s^{\ell}}(n)$ define a bijection on $\Lambda^+_{s^{\ell}}(n)$. We denote by $\la^+$ the unique element in $\Lambda^+_{s^{\ell}}(n)$ such that $(\la^+)^- = \la$.
\end{rem}
\begin{rem} \label{ProjInje}
Since the translation functors are exact and both bi-adjoint to each other, they preserve projective and injective modules. In particular, $Z_{\la}U({t_{\la}})$ is both a projective cover and an injective hull.
\end{rem}

Let $\tau : \mf{g} \rightarrow \mf{g}$ be the anti-automorphism on $\mf{g}$ defined by $\tau(e_{ij}) = e_{ji}$ and $\tau(\overline{e}_{ij}) = \overline{e}_{ji}$ (see, e.g., \cite[Example 7.10]{Br3}). Then $\tau$ induces a contravariant auto-equivalence on $\mc{O}$ and $\mc{F}$ (see, e.g., \cite[Section 3.2]{Hum08} and \cite[Section 2.1]{Ger98}). For a given $M\in \mc{F}$, let $M^{\tau}$ be the image of $\tau$. Since $\text{ch}L(\la)^{\tau} = \text{ch}L(\la)$, we have $L(\la)^{\tau} \in \{L(\la), \Pi L(\la)\}$. Note that $\ell(\la) = n$ for each $\la \in \Lambda_{s^{\ell}}(n) $, we thus have the following lemma:
\begin{lem} \emph{(} \cite[Lemma 7]{Fr} \emph{)} \label{TauOfL}
Let $n$ be even and $\la \in \Lambda_{s^{\ell}}(n) $. Then we have
\begin{align}
L(\la)^{\tau} \cong \left\{ \begin{array}{ll}  \Pi L(\la) \text{ if } n \equiv 2 \text{ mod 4},
\\ L(\la) \text{ if } n \equiv 0 \text{ mod 4}.
 \end{array} \right.
 \end{align}
\end{lem}

The following corollary is an immediate consequence by Lemma \ref{TauOfL}.
\begin{cor} \label{CorForTauOfL}
Let $M\in \mc{F}$. If $n \equiv 0 \text{ mod 4}$ then $M^{\tau}$ and $M$ have identical set of composition factors. If $n \equiv 2 \text{ mod 4}$ then $M^{\tau}$ and $\Pi M$ have identical set of composition factors.
\end{cor}

For a given $M\in \mc{O}$, denote by $\text{rad}M$  and $\text{soc}M$ the radical and socle of $M$, respectively. We now can prove the following BGG reciprocity.
\begin{lem} \label{BGGReciprocity} \emph{(}BGG reciprocity\emph{)}
 Let $\la, \mu \in \Lambda^+_{s^{\ell}}(n)$. Then we have
\begin{align}
(P(\la): \Pi^i K(\mu)) = [\Pi^i K(\mu):L(\la)], \label{BGGforProjective}.
\end{align}
for $i=0,1$.
\end{lem}
\begin{proof}
We first assume $n$ is even. Let $\la', \mu' \in \Lambda^+_{s^{\ell}}(n)$ be arbitrary. Recall that $L(\la') \not \cong \Pi L(\la')$ in this case.
By a similar arguments of consideration on (highest) weights as in \cite[Theorem 3.3.(c),(d)]{Hum08} we first have that
\begin{align}
\text{Ext}_{\mc{F}}(\Pi^i K(\la'),K(\mu')^{\tau})=0, \label{ExtInBGG} \\
\text{dimHom}_{\mc{F}}(\Pi^i K(\la'),K(\mu')^{\tau}) \neq 0 \text{ implies that } \la' =\mu'. \label{HomInBGG}
\end{align}
 Furthermore, by Lemma \ref{TauOfL} we have
 \begin{align}\text{soc}K(\mu')^{\tau} = L(\mu')^{\tau} \cong \left\{ \begin{array}{ll}  \Pi L(\mu') \text{ if } n \equiv 2 \text{ mod 4},
\\ L(\mu') \text{ if } n \equiv 0 \text{ mod 4}.
 \end{array} \right.  \end{align}
By a similar proof as in \cite[Theorem 3.3(c)]{Hum08}, we may conclude that
\begin{align}
\text{dimHom}_{\mc{F}}(\Pi^i K(\la'),(K(\mu'))^{\tau}) = \left\{ \begin{array}{lll} 0 \text{ if } \la' \neq \mu',
\\ i \text{ if } \la' = \mu' \text{ and } n \equiv 2 \text{ mod 4},
\\ 1-i \text{ if } \la' = \mu' \text{ and } n \equiv 0 \text{ mod 4}.
 \end{array} \right.
\end{align}
\begin{align}
\text{dimHom}_{\mc{F}}(\Pi^iK(\la'),(\Pi K(\mu'))^{\tau}) = \left\{ \begin{array}{lll} 0 \text{ if } \la' \neq \mu',
\\ i \text{ if } \la' = \mu' \text{ and } n \equiv 0 \text{ mod 4},
\\ 1-i \text{ if } \la' = \mu' \text{ and } n \equiv 2 \text{ mod 4}.
 \end{array} \right.
\end{align}
%there is a non-zero homomorphism $\phi : K(\la') \rightarrow K(\la')^{\tau}$ (resp. $\phi : K(\la') \rightarrow (\Pi K(\la'))^{\tau}$) if and only if $n \equiv 0 \text{ mod 4}$ (resp. $n \equiv 2 \text{ mod 4}$ ).

Since $P(\la)$ has a flag of parabolic Verma modules, as a conclusion, we have
\begin{align}
(\Pi^i P(\la): K(\mu))= \left\{ \begin{array}{ll} \text{dimHom}_{\mc{F}}(\Pi^i P(\la), (\Pi K( \mu))^{\tau})  \text{ if } n \equiv 2-2i \text{ mod 4},\\
\text{dimHom}_{\mc{F}}(\Pi^i P(\la), K(\mu)^{\tau})  \text{ if } n \equiv 2i \text{ mod 4}.
 \end{array} \right.
\end{align}
Recall that $\text{dimHom}_{\mc{F}}(\Pi^i P(\la), M) = [M:\Pi^i L(\la)]$ for all $M \in \mc{F}$ and $\la \in \Lambda^+_{s^{\ell}}(n)$ (see, e.g, \cite[Section 3.9]{Hum08}). By Corollary \ref{CorForTauOfL}, we have $[(\Pi K(\mu))^{\tau}: \Pi^i L(\la)] = [K(\mu): \Pi^i L(\la)]$  for $ n \equiv 2 \text{ mod 4}$ (resp. $[K(\mu)^{\tau}: \Pi^i L(\la)] = [K(\mu): \Pi^i L(\la)]$  for $ n \equiv 0 \text{ mod 4}$). The proof of this lemma follows provided that $n$ is even.

Recall that $L(\la')\cong \Pi L(\la')$ if $n$ is odd. By a similar argument, the proof for odd $n$ can be obtained. This completes the proof.
\end{proof}

%% old version for BGG-reciprocity
%\begin{lem} \label{BGGReciprocity} \emph{(}BGG reciprocity\emph{)} Let $\la, \mu \in \Lambda_{s^{\ell}}(n) $ with $\sharp \la=\sharp \mu =1$. Then $P(\la) = U(\la^+)$ and \begin{align} [P(\la): K(\mu)] = [K(\mu):L(\la)],\label{BGGforProjective}  \\  [U(\la):K(\mu)] = [K(-\omega_0 \mu):L(-\omega_0 \la)]. \label{BGGforTilting} \end{align} \end{lem} \begin{proof}
% Let $\la', \mu' \in \Lambda_{s^{\ell}}(n) $.  Note that both $L^0(\la'),L^0(\mu')$ are projective in $\mc{O}^{\mf{l}+\mf{n}}$. By a standard argument (see, e.g., \cite[Lemma 3.6.1]{Ger98}), we can obtained the following isomorphisms
%\begin{align}
%\text{Ext}^i_{\mc{O}^{\mf{p}}}(K(\la'),K^-(\mu')) \cong \text{Ext}^i_{\mc{O}^{\mf{l}+\mf{n}}}(L^0(\la'),L^0(\mu')) \cong \left\{ \begin{array}{ll}   \mathbb{C} \text{, in case $i =0$ and $\la' =\mu'$}, \\ 0 \text{, otherwise }. \end{array} \right. \end{align}Note that  $\text{dimHom}_{\mc{O}^{\mf{p}}}(P(\la'), M) = [M:L(\la')]$ for all $M \in \mc{O}^{\mf{p}}$ and $\la' \in \Lambda_{s^{\ell}}(n)$ (see, e.g, \cite[Section 3.9]{Hum08}). Since $P(\la)$ has a Verma flag, we have \begin{align}[P(\la):K(\mu)] = \text{dimHom}_{\mc{O}^{\mf{p}}}(P(\la), K^-(\mu)) = [K^-(\mu):L(\la)]. \end{align}Then the formula \eqref{BGGforProjective} follows from $\text{ch}K^-(\mu) = \text{ch}K(\mu)$. Note that $[P(\mu)] =[K(\mu)]+[K(\mu^+)]$, in $\mc{K}(\mc{F})$, for all $\mu\in \Lambda_{\la}$ and $(-\omega_0\la)^{-} = -\omega_0(\la^+)$ by Lemma \ref{ProCovFromBrsAlgorithm}. Therefore we have $P(\mu) =U(\mu^+)$ and so \eqref{BGGforTilting} follows. \end{proof}

Let $\la \in \Lambda^+ _{s^{\ell}}(n)$ with $\sharp{\la} =1$. Define $\Lambda_{\la}:=\{\la^i| i\in \mathbb{Z}\}$ by the recursive relation $\la^{i+1} = (\la^i)^+$ and $\la^0 = \la$. We define irreducible modules $L_{\la^i}$ and parabolic Verma modules $K_{\la^i}$ for $i\in \mathbb{Z}$ as follows. First, $L_{\la^0}:=L(\la)$ and $K_{\la^0}:= K(\la)$. Then $L_{\la^i}$ for $i\neq 0$ is defined by the recursive relation coming from the short exact sequences:
\begin{align} \label{IrrChoicesInVerma}
0\rightarrow L_{\la^{i-1}} \rightarrow K_{\la^i} \rightarrow L_{\la^i} \rightarrow 0,
\end{align}
for $i\in \mathbb{Z}$. Let $P_{\la^i}$ be the projective cover of $L_{\la^i}$. Then by Lemma \ref{BGGReciprocity}, we have short exact sequence
\begin{align} \label{KacChoicesInProj}
0\rightarrow K_{\la^{i+1}} \rightarrow P_{\la^i} \rightarrow K_{\la^i} \rightarrow 0,
\end{align}
for $i\in \mathbb{Z}$. Let  $\mc{F}_{\la}$ be the block of $\mc{F}$ containing $L(\la)$. Namely, it is the Serre subcategory generated by the set of vertices in the connected component of the Ext-quiver for $\mc{F}$ containing $L(\la)$.

%% Separaction Lemma
\begin{cor} \label{Separaction Lemma}
Let  $\la \in \Lambda^+_{s^{\ell}}(n)$ with $\sharp \la =1$. Then $L(\la)$ and $\Pi L(\la)$ are in different blocks if and only if $n$ is even. Furthermore, $\mc{F}_{\la}$ is the Serre subcategory of $\mc{F}$ generated by $\{L_{\mu}| \mu \in \Lambda_{\la}\}$.
\end{cor}
\begin{proof}
By Lemma \ref{ProCovFromBrsAlgorithm} and Lemma \ref{BGGReciprocity}, the set $\{L_{\la^i},L_{\la^i}, L_{\la^{i+1}},L_{\la^{i-1}}\}$ is the set of all composition factors of $P_{\la^i}$. Recall that
\begin{align} \label{ExtByRadicalFormula}
\text{Ext}_{\mc{F}}(L_{\la^i},\Pi^jL(\mu)) \cong \text{Hom}_{\mc{F}}(\text{rad}P_{\la^i}/\text{rad}^2P_{\la^i}, \Pi^jL(\mu)),
\end{align}
for all $i\in \mathbb{Z}, j=0,1$ and $\mu \in \Lambda^+_{s^{\ell}}(n)$. This means that $\{L_{\la^i}| i\in\mathbb{Z}\}$ is the complete set of irreducible objects of $\mc{F}_{\la}$ by \eqref{IrrChoicesInVerma}. Since for each $i\in \mathbb{Z}$ the object $L_{\la^i}$ is of highest weight $\la$ if and only if $i=0$ (i.e. $\la^i =\la$), we may conclude that $\Pi L(\la) \not \in \mc{F}_{\la}$ if and only if $n$ is even by Lemma \ref{ClarificationOfParityOfL}.
\end{proof}

\begin{example} \label{q2Example} (The Ext-quiver for $\mf{q}$(2))
Let $n=2$. Let $s\notin \mathbb{Z}/2$, $\alpha : =\varepsilon_1 -\varepsilon_2$, $\widetilde{s} \in s+\mathbb{Z}$ and  $\la := \widetilde{s} \alpha \in \Lambda^+_{s^{1}}(2)$. In this case, we have $K(\la) = M(\la)$ and $I_{\la} = \text{Ind}_{\mf{h}'}^{\mf{h}}\mathbb{C}v_{\la}$ is a two dimensional irreducible $\mf{h}$-module with $\mf{h}' = \mf{h}_{\bar{0}} \oplus \mathbb{C}(\overline{h_1}+\overline{h_2})$, where $\mf{h}'_{\bar{1}}$ acts on $\mathbb{C}v_{\la}$ trivially. For a given irreducible $\mf{h}$-module $V$ of $\mf{h}_{\bar{0}}$-weight $\la-\alpha$, we have that $V \cong I_{\la -\alpha}$  (resp. $V \cong \Pi I_{\la-\alpha}$ ) if and only $(\overline{h}_1+\overline{h}_2)V_{\bar{0}} =0$  (resp. $(\overline{h}_1 -\overline{h}_2)V_{\bar{0}} =0$). It is not hard to compute that  $u:=e_{21}\overline{h}_1v_{\la} - \widetilde{s}\overline{e}_{21}v_{\la}$ is a singular vector in $M(\la)$ with $\overline{h}_1u =-\overline{h}_2u$ (see, e.g., \cite[Lemma 2.44(2)]{CW}). Note that $\overline{u} =\overline{1}$,  therefore we have short exact sequence
 \begin{align}
 0 \rightarrow  \Pi L(\la - \alpha) \rightarrow K(\la) \rightarrow L(\la) \rightarrow 0. \label{q2KacStructure}
\end{align}
Thus,  $\text{Ext}_{\mc{F}}(L(\la),  \Pi L(\la-\alpha) ) \neq 0$. By Lemma \ref{BGGReciprocity} we may conclude that $P(\la)$ has a Verma flag
\begin{align}
0 \rightarrow \Pi  K(\la+\alpha) \rightarrow P(\la) \rightarrow K(\la) \rightarrow 0,
\end{align}
 and so $\{L(\la), \ \ L(\la), \ \ \Pi L(\la+\alpha),\ \ \Pi  L(\la-\alpha)\}$ is the complete set of composition factors of $P(\la)$. Therefore the set of objects of $\mc{F}_{\la}$ is $\{ \Pi^k L(\la + k\alpha)| k\in \mathbb{Z}\}$. Furthermore, if we apply $\Pi \circ \tau$ to the short exact sequence \eqref{q2KacStructure}, then it follows that $\text{Ext}_{\mc{F}}( L(\la-\alpha), L(\la) ) \neq 0$. Replace $\la$ by arbitrary $\la +k\alpha$, we may conclude that the Ext-quivers of $\mc{F}_{\la}$ and that of the principal block $(\mc{F}_{1|1})_0$ for $\mf{gl}(1|1)$ are the same.
\end{example}

By a similar argument, we may generalize results in Example \ref{q2Example}. More precisely, let $\widetilde{\tau}:= \Pi \circ \tau$ (resp. $\widetilde{\tau}:= \tau$) if $n \equiv 2 \text{ mod } 4$ (resp. $n \equiv 0 \text{ mod } 4$). Let $\la \in \Lambda^+_{s^{\ell}}(n)$ with $\sharp \la =1$. By applying $\widetilde{\tau}$ to the non-trivial short exact sequence \eqref{IrrChoicesInVerma}, it follows from \eqref{ExtByRadicalFormula} that $\text{rad} P(\la ) / \text{rad}^2P(\la) \cong L_{\la^+}\oplus L_{\la^-}$. As a consequence, we obtain the following lemma.

\begin{lem} \label{StructureOfProj}
Let $\la \in \Lambda^+_{s^{\ell}}(n)$ with $\sharp \la =1$. Then there are exactly four distinct proper submodules of $P(\la)$:
\begin{align}
\ A_{\la} \cong K_{\la^+}, \ \ B_{\la}, \ \ \emph{rad}P(\la)= A_{\la} + B_{\la}, \ \ \emph{rad}^2P(\la)= \emph{soc}P(\la)\cong L_{\la}.
\end{align}
Furthermore, we have the following short exact sequences:
\begin{align}
0 \rightarrow A_{\la} \rightarrow \emph{rad}P(\la) \rightarrow L_{\la^{-}} \rightarrow 0, \label{RadOverA} \\
0 \rightarrow B_{\la} \rightarrow \emph{rad}P(\la) \rightarrow L_{\la^{+}} \rightarrow 0, \label{RadOverB}\\
0 \rightarrow L_{\la}  \rightarrow A_{\la} \rightarrow L_{\la^{+}} \rightarrow 0, \label{StruOfA} \\
0 \rightarrow L_{\la}  \rightarrow B_{\la} \rightarrow L_{\la^{-}} \rightarrow 0. \label{StruOfB}
\end{align}
\end{lem}

 We are now in a position to state the main theorem in this section.
\begin{thm} \label{EquivalenceThm}
Let $\la \in \Lambda^+_{s^{\ell}}(n)$ with $\sharp \la =1$. Then $\mc{F}_{\la}$ is equivalent to $(\mc{F}_{{\ell | n-\ell}})_{\la^{\mf{a}}}$.
\end{thm}

\begin{proof}
 Our  goal is to prove that the endomorphism algebra
$\text{End}_{\mc{F}_{\la}}(\oplus_{\mu\in\Lambda_{\la}} P_{\mu})$ of projective generator $\oplus_{\mu\in\Lambda_{\la}} P_{\mu}$ for $\mc{F}_{\la}$ is isomorphic to $(K^{\infty}_{1})^{\text{op}}$. We fix $\mu =\la^i \in \Lambda_{\la}$. It follows from Lemma \ref{StructureOfProj} that $\text{End}_{\mc{F}_{\la}}(P_{\mu}, P_{\mu'}) =0$ if $\mu' \not \in \{ \mu^{+}, \mu, \mu^-  \}$. Since $\text{dimHom}_{\mc{F}_{\la}}(L_{\mu}, L_{\mu'}) =\delta_{\mu,\mu'} $ for all $\mu,\mu' \in \Lambda_{\la}$, it is not hard to prove that (see, e.g, \cite[Section 3.9]{Hum08})
\begin{align} \label{DimFormula}
\text{dimHom}_{\mc{F}_{\la}}(P_{\mu}, P_{\mu'}) = [P_{\mu'}:L_{\mu}] = \left\{ \begin{array}{ll}  1 \text{ in case $\mu'\in\{\mu^{+}, \mu^{-}\}$  },
\\ 2 \text{ in case $\mu' =  \mu$ }.
 \end{array} \right.
\end{align}
We now construct bases for $\text{Hom}_{\mc{F}_{\la}}(P_{\mu}, P_{\mu'})$, for $\mu' \in\{\mu ,\mu^{+}, \mu^{-}\}$:

 (1). A Basis for $\text{Hom}_{\mc{F}_{\la}}(P_{\mu}, P_{\mu})$:
 First note that the canonical epimorphism $P_{\mu}\rightarrow L_{\mu}$ gives an endomorphism of $\widetilde{z}_i\in  \text{End}_{\mc{F}_{\la}}P_{\mu}$ by mapping $P_{\mu}$ onto $\text{soc}P_{\mu} \subset P_{\mu}$. Let $1_i \in \text{End}_{\mc{F}_{\la}}P_{\mu}$ be the identity map, then we may conclude that $\text{End}_{\mc{F}_{\la}}P_{\mu}$ is generated by $1_i, \widetilde{z}_i$ by \eqref{DimFormula}.

  (2). A Basis for $\text{Hom}_{\mc{F}_{\la}}(P_{\mu}, P_{\mu^-})$:
  Next note the composition of homomorphisms
  \begin{align}
  \widetilde{ y}_{i-1}: P_{\mu} \rightarrow P_{\mu}/ A_{\mu} \cong  A_{\mu^-} \hookrightarrow P_{\mu^-}, \label{ComputeRelationY}
  \end{align} gives a non-zero element $\widetilde{ y}_{i-1}\in \text{Hom}_{\mc{F}_{\la}}(P_{\mu}, P_{\mu^-})$, and so $\mathbb{C}\widetilde{y}_{i-1} = \text{Hom}_{\mc{F}_{\la}}(P_{\mu}, P_{\mu^-})$ by \eqref{DimFormula} again.
  %and so $\widetilde{y}_{i-1}$ is a generator $\text{Hom}_{\mc{F}_{\la}}(P_{\mu}, P_{\mu^-})$ by \eqref{DimFormula} again.

   (3). A Basis for $\text{Hom}_{\mc{F}_{\la}}(P_{\mu}, P_{\mu^+})$:
  Note that $\text{dimHom}_{\mc{F}_{\la}}(P_{\mu}, P_{\mu^+}) =1$ by \eqref{DimFormula}. Let $\widetilde{x}_i \in \text{Hom}_{\mc{F}_{\la}}(P_{\mu}, P_{\mu^+})$ be a non-zero element. Note that each composition factor of $\widetilde{x}_i(P_{\mu})$ lies in $\{ L_{\mu^{++}}, L_{\mu}, L_{\mu^+} \}$ because $\widetilde{x}_i(P_{\mu}) \subset \text{rad}P_{\mu^+}$.
Since every non-trivial quotient of $P_{\mu}$ has composition factor $L_{\mu}$, we may conclude that $\widetilde{x}_i(P_{\mu}) = B_{\mu^+}$. Consequently,   $\widetilde{x}_i$ can be expressed as the following composition of homomorphisms
  \begin{align}
 \widetilde{x}_i: P_{\mu} \rightarrow P_{\mu}/ B_{\mu} \cong B_{\mu^+} \hookrightarrow P_{\mu^+}. \label{ComputeRelationX}
  \end{align}

  It is not hard to compute the relations $ \widetilde{x}_j \widetilde{x}_{j+1} = \widetilde{y}_j \widetilde{y}_{j+1} = 0$ for all $j\in \mathbb{Z}$ by  \eqref{ComputeRelationY} and \eqref{ComputeRelationX}. Therefore we may conclude that $ \widetilde{z}_ic = c\widetilde{z}_i = \widetilde{y}_i\widetilde{y}_j =\widetilde{x}_i\widetilde{x}_j =0  $ for all $i,j\in \mathbb{Z}, c\in \{\widetilde{x}_i,\widetilde{y}_i\}_{i\in \mathbb{Z}}$. Finally, we note that $\widetilde{y}_i\widetilde{x}_i$ and $\widetilde{x}_i\widetilde{y}_i$ can be expressed as the following composition of homomorphisms
\begin{align}
\widetilde{y}_i \widetilde{x}_i :\ \ P_{\mu} \rightarrow P_{\mu} / B_{\mu} \cong B_{\mu^+}  \rightarrow  \frac{A_{\mu^+}+ B_{\mu^+}}{ A_{\mu^+}} \cong \text{soc}P_{\mu} \subset P_{\mu}, \\
\widetilde{x}_i \widetilde{y}_i :\ \ P_{\mu^+} \rightarrow P_{\mu^+} / A_{\mu^+} \cong A_{\mu}  \rightarrow  \frac{A_{\mu}+ B_{\mu}}{ B_{\mu}} \cong \text{soc}P_{\mu^+} \subset P_{\mu^+}.
\end{align}
It is not hard to see that $\widetilde{y}_i \widetilde{x}_i = \widetilde{z}_i$ and $\widetilde{x}_i \widetilde{y}_i =  \widetilde{z}_{i+1}$ for all $i \in \mathbb{Z}$. Therefore, we have an isomorphism from $\text{End}_{\mc{F}_{\la}}(\oplus_{\mu\in\Lambda_{\la}} P_{\mu})$ to $(K^{\infty}_{1})^{\text{op}}$ sending $\widetilde{x}_i, \widetilde{y}_i, \widetilde{z}_i$ to $x_i,y_i,z_i$, respectively. This completes the proof.
\end{proof}

\end{document}